\documentclass[english,11pt,a4paper]{article}

\usepackage{graphicx}
\usepackage{amsmath}
\usepackage{amsthm}
\usepackage{amssymb}
\usepackage{amsfonts}
\usepackage{latexsym}

\newcommand{\sign}{{\rm sign}\,}

\newcommand{\PP}{{\mathcal P}}

\newcommand{\ze}{\zeta}
\newcommand{\zs}{\zeta(s)}
\newcommand{\DD}{{\mathcal D}}
\newcommand{\GG}{{\mathcal G}}

\newcommand{\G}{{\mathcal G}}
\newcommand{\A}{{\mathcal A}}

\newcommand{\Z}{{\mathcal Z}}

\newcommand{\RR}{{\mathbb R}}
\newcommand{\CC}{{\mathbb C}}

\newcommand{\NN}{{\mathbb N}}

\newcommand{\N}{{\mathcal N}}
\newcommand{\R}{{\mathcal R}}

\newcommand{\SSS}{{\mathcal S}}

\newcommand{\MM}{{\mathcal M}}

\newcommand{\D}{{\Delta}}

\newcommand{\al}{{\alpha}}
\newcommand{\ve}{{\varepsilon}}
\newcommand{\de}{\delta}
\newcommand{\si}{\sigma}

\newcommand{\zp}{\zeta_{\mathcal P}}
\newcommand{\psip}{\psi_{\mathcal P}}
\newcommand{\Dp}{\Delta_{\mathcal P}}

\newcommand{\pip}{\pi_{\mathcal P}}
\newcommand{\PiP}{\Pi_{\mathcal P}}

\newcommand{\srs}{\sum_{\rho\in\SSS}}

\newtheorem{theorem}{Theorem}

\newtheorem{lemma}{Lemma}
\newtheorem{proposition}{Proposition}
\newtheorem{definition}{Definition}

\theoremstyle{definition}

\newtheorem{remark}{Remark}

\begin{document}

\title{Oscillation of the remainder term in the prime number theorem of Beurling, "caused by a given $\zeta$-zero"}



\author{Szil\' ard Gy. R\' ev\' esz}




\maketitle

\begin{abstract} Continuing previous studies of the Beurling zeta function, here we prove two results, generalizing long existing knowledge regarding the classical case of the Riemann zeta function and some of its generalizations.

First, we address the question of Littlewood, who asked for explicit oscillation results provided a zeta-zero is known. We prove that given a zero $\rho_0$ of the Beurling zeta function $\zeta_{\PP}$ for a given number system generated by the primes $\PP$, the corresponding error term $\Delta(x):=\psi_{\PP}(x)-x$, where $\psi_{\PP}(x)$ is the von Mangoldt summatory function shows oscillation in any large enough interval, as large as $\frac{\pi/2-\ve}{|\rho_0|}x^{\Re \rho_0}$.

The somewhat mysterious appearance of the constant $\pi/2$ is explained in the study. Finally, we prove as the next main result of the paper the following: given $\ve>0$, there exists a Beurling number system with primes $\PP$, such that $|\Delta(x)| \le \frac{\pi/2+\ve}{|\rho_0|}x^{\Re \rho_0}$.

In this second part a nontrivial construction of a low norm sine polynomial is coupled by the application of the wonderful recent prime random approximation result of Broucke and Vindas, who sharpened the breakthrough probabilistic construction due to Diamond, Montgomery and Vorhauer.
\end{abstract}



\medskip
{\bf Author information.} Alfréd Rényi Institute of Mathematics\\
Reáltanoda utca 13-15, 1053 Budapest, Hungary \\
{\tt revesz.szilard@renyi.hu}



\section{Introduction}
This work deals with Beurling's theory of generalized integers and primes.
The theory fits well to the study of several mathematical
structures. A vast field of applications of Beurling's theory is
nowadays called \emph{arithmetical semigroups}, which are
described in detail e.g. by Knopfmacher, \cite{Knopf}.
For important examples where the theory is of relevance see
Knopfmacher's book \cite{Knopf}, pages 11-22.

Here $\G$ is a unitary, commutative semigroup, with a countable set
$\PP$ of indecomposable generators, called the \emph{primes} of $\G$,
which freely generate the whole of $\G$: i.e.,
any element $g\in \G$ can be (essentially, i.e. up to order of
terms) uniquely written in the form $g=p_1^{k_1}\cdot \dots \cdot
p_m^{k_m}$: two (essentially) different such expressions are
necessarily different as elements of $\G$, while each element has
its (essentially) own unique prime decomposition.

Moreover, there is a \emph{norm} $|\cdot|~: \G\to \RR_{+}$ so that
the following hold. First, the image of $\G$, $|\G|\subset
\RR_{+}$ is \emph{locally finite} (this property sometimes being also called "discrete"), i.e. any finite interval of $\RR_{+}$ can contain the norm of only a finite number of elements of $\G$; thus the function
\begin{equation}\label{Ndef}
{\N}(x):=\# \{g\in \G~:~ |g| \leq x\}
\end{equation}
exists as a finite, nondecreasing, right continuous, nonnegative integer valued
function on $\RR_{+}$. Second, the norm is multiplicative, i.e. $|g\cdot h| = |g| \cdot
|h|$; it follows that for the unit element $e$ of $\G$ $|e|=1$, and
that all other elements $g \in \G$ have norms strictly larger than 1.


In this work we will assume the so-called \emph{"Axiom A"} (in its normalized form to $\delta=1$) of
Knopfmacher, see pages 73-79 of his fundamental book \cite{Knopf}.

\begin{definition} It is said that ${\N}$ (or, loosely speaking, $\ze$)
satisfies \emph{Axiom A} -- more precisely, Axiom
$A(\kappa,\theta)$ with the suitable constants $\kappa>0$ and
$0<\theta<1$ -- if we have\footnote{The usual formulation uses the more natural version $\R(x):= \N(x)-\kappa x$. However, our version is more convenient with respect to the initial values at 1, as we here have $\R(1-0)=0$. All respective integrals of the form $\int_1$ will be understood as integrals from $1-0$, and thus we can avoid considering endpoint values in the partial integration formulae. Alternatively, we could have taken also $\N(x)$ left continuous, and integrals from 1 in the usual sense: also with this convention we would have $\R(1)=0$.} for the remainder term
$$
\R(x):= \N(x)-\kappa (x-1)
$$
the estimate
\begin{equation}\label{Athetacondi} \left| \R(x) \right|  \leq A x^{\theta} \quad (x \geq 1 ~ \textrm{arbitrary}).
\end{equation}
\end{definition}
The Beurling zeta function is defined as the Mellin transform of $\N(x)$, i.e.
\begin{equation}\label{zetadef}
\ze(s):=\ze_{\G}(s):=\MM(\N)(s):=\int_1^{\infty} x^{-s} d\N(x) = \sum_{g\in\G} \frac{1}{|g|^s}.
\end{equation}
If only $\N(x)=O(x^C)$, the series converges absolutely and locally uniformly in the halfplane $\Re s> C+1$, moreover, its terms can be rearranged to provide the \emph{Euler product formula}
\begin{equation}\label{Euler}
\ze_{\G}(s)=\prod_{p\in\PP} \left(\frac{1}{1-|p|^{-s}}\right).
\end{equation}
In particular, if $\N(x)=O(x^{1+\ve})$ for all $\ve>0$, then $\zeta_\GG$ is absolutely convergent in $\Re s>1$, it cannot vanish there--as is clear from \eqref{Euler}--moreover, $|\zs| \ge 1/\zeta(\si)$ ($\si:=\Re s$). Furthermore, under Axiom A it admits a meromorphic, essentially analytic continuation
$\kappa\frac{1}{s-1}+\int_1^{\infty} x^{-s} d\R(x)$ up to $\Re s >\theta$
with only one, simple pole at 1. For an analysis of the finer behavior of the number of primes $\pi_{\PP}(x):=\sum_{p\in \PP;~|p|\le x} 1$--as in the classical case of $\GG=\NN$--the location of the zeroes of $\zs$ in the "critical strip" $\theta<\Re s\le 1$ is decisive, as we will see.

The Beurling zeta function \eqref{zetadef} can be used to express
the generalized von Mangoldt function
\begin{equation}\label{vonMangoldtLambda}
\Lambda (g):=\Lambda_{\G}(g):=\begin{cases} \log|p| \quad & \textrm{if}\quad g=p^k,
~ k\in\NN ~~\textrm{with some prime}~~ p\in\PP\\
0 \quad & \textrm{if}\quad g\in\G ~~\textrm{is not a prime power in} ~~\G
\end{cases}
\end{equation}
as coefficients of the logarithmic derivative of the zeta function
\begin{equation}\label{zetalogder}
-\frac{\zeta'}{\zeta}(s) = \sum_{g\in \G} \frac{\Lambda(g)}{|g|^s}.
\end{equation}

The Beurling theory of generalized primes is mainly concerned with
the analysis of the summatory function
\begin{equation}\label{psidef}
\psi(x):=\psi_{\G}(x):=\sum_{g\in \G,~|g|\leq x} \Lambda (g).
\end{equation}
As is well-known, $\psi(x)$ is essentially $\pip(x)$, apart from an inessential logarithmic weight and a smaller order contribution from higher prime powers. Therefore, the asymptotic relation $\psi(x)\thicksim x$ is equivalent to say that $\pi(x)\thicksim {\rm li}(x):=\int_2^x du/\log u$ or $x/\log x$, and is thus termed as \emph{the Prime Number Theorem} (PNT). Equivalently, we can also formulate this by use of the "error term in the prime number formula", for which the standard notation is
\begin{equation}\label{Deltadef}
\Delta(x):=\Delta_{\G}(x):=\psi(x)-x.
\end{equation}
Then PNT is thus the statement that $\Delta(x)=o(x)$. The so-called "Chebyshev bounds" $x \ll \psi(x) \ll x$, weaker than PNT,  mean that there exist positive constants $0<c_1<c_2<\infty$ with $c_1x \le \psi(x) \le c_2 x$. Extending hundred years old knowledge for the natural numbers $\NN$ as $\GG$ and the corresponding Riemann zeta function, much study was devoted to describe, what conditions are necessary resp. sufficient for PNT or the Chebyshev bounds to hold. As in the classical case, it was clarified that for the Chebyshev bounds it suffices to control the behavior of the Beurling zeta function in the convergence halfplane $\si:=\Re s>1$, and several conditions were found to ensure these basic inequalities, see \cite{Vindas12, Vindas13, Zhang93, DZ-13-2, DZ-13-3}. For the PNT the analysis also followed classical lines, demonstrating that only proper behavior of the Beurling zeta function $\zs$ in and on the boundary of the convergence halfplane is needed for the PNT to hold. In this wide generality, however, when no analytic (meromorphic) continuation is assumed (so that in particular Axiom A is not postulated), delicate studies revealed a fine connection of "nicety" of the boundary function on the one hand, and validity of the PNT on the other hand \cite{Beur, K-98, DebruyneVindas-PNT, DSV, DZ-17, Zhang15-IJM, Zhang15-MM}. Also an interesting new question, which simply does not arise in the classical case, is the converse direction: assuming some form of the PNT (assuming it with some control on the error $\Delta(x)$), derive density results for the number of integers $\N(x)$ \cite{Diamond-77, K-17, DebruyneVindas, Schlage-PuchtaVindas}.

There are other studies related to the Beurling PNT in the literature. In particular, some rough (as compared to our knowledge in the natural prime number case) estimates and equivalences were worked out in the analysis of the connection between $\zeta$-zero distribution and the behavior of the error term $\Delta(x)$. One of the deep results\footnote{In the course of his proof Kahane also proves that the zeta function has $O(T)$ zeros (counted with multiplicity) on the line segment $\{s = a+it; 0 < t \le T\}$ with any $a>\max(1/2,\theta)$, i.e. the zero counting function has finite upper density on the vertical line through $a$. To the best of our knowledge this is the only result of a zero-density estimate feature preceding our recent study \cite{Rev-D}.}
in this direction is the extension (apart from a minor loss in the precision regarding the logarithmic terms) of the classical oscillation result $\pi(x)-{\rm li}(x) =\Omega_{\pm}(\sqrt{x}\log\log\log x/\log x)$ of Littlewood \cite{Littlewood-CR} to the Beurling context \cite{K-99}. Further, so-called $(\alpha,\beta)$ systems and $(\alpha,\beta,\gamma)$ systems were defined \cite{H-20}, with these parameters denoting the "best possible" exponents in estimating the error terms $ \Delta(x), \R(x)$ and the summatory function $M_{\GG}(x)$ of the Beurling version of the M\"obius function $\mu_{\GG}(g)$; in particular, Hilberdink showed that the two largest of these three parameters have to be at least 1/2 and must match \cite{H-5}, \cite{H-20}. Oscillation order of the generalized M\"obius summatory function and even more general arithmetical functions are also treated up to recent times \cite{H-10, DZ-17, DebruyneDiamondVindas, BrouckeVindas, PintzMP, DeMaVi}.

A natural, but somewhat different direction, going back to Beurling himself, is the study of analogous questions in case the assumption of Axiom A is weakened to e.g. an asymptotic condition on $\N(x)$ with a product of $x$ and a sum of powers of $\log x$, or sum of powers of $\log x$ perturbed by almost periodic polynomials in $\log x$, or $N(x)-cx$ periodic, see \cite{Beur}, \cite{Zhang93}, \cite{H-12}, \cite{RevB}.

Apart from generality and applicability to e.g. distribution of
prime ideals in number fields, the interest in the Beurling theory
was greatly boosted by a construction of Diamond, Montgomery
and Vorhauer \cite{DMV}. They basically showed that under Axiom A
the Riemann Hypothesis (RH for short from here on) may still fail; moreover, nothing better
than the most classical \cite{V} zero-free region and error term  of
\begin{equation}\label{classicalzerofree}
\zeta(s) \ne 0 \qquad \text{whenever}~~~ s=\sigma+it, ~~ \sigma >
1-\frac{c}{\log t},
\end{equation}
and
\begin{equation}\label{classicalerrorterm}
\psi(x)=x +O(x\exp(-c\sqrt{\log x})
\end{equation}
follows from \eqref{Athetacondi} at least if $\theta>1/2$.

Therefore, Vinogradov mean value theorems on trigonometric sums
and many other stuff are certainly irrelevant in this generality,
and for Beurling zeta functions a careful reconsideration of the
combination of "ancient-classical" methods and "elementary"
arguments can only be implemented.

After the Diamond-Montgomery-Vorhauer paper \cite{DMV}, better and better examples were constructed for arithmetical semigroups with very "regular" (such as satisfying RH and error estimates $\psi(x)=x+O(x^{1/2+\varepsilon})$) and very "irregular" (such as having no better zero-free regions than \eqref{classicalzerofree} and no better asymptotic error estimates than \eqref{classicalerrorterm}) behavior and zero- or prime distribution, see e.g. \cite{H-15}, \cite{BrouckeDebruyneVindas}, \cite{DMV}, \cite{H-5}, \cite{Zhang7}. For a throughout analysis of these directions as well as for much more information the reader may consult the monograph \cite{DZ-16}.

In sum, in contrast with the classical natural number system, when it is generally believed that the Riemann Hypothesis holds true, in the generality of arithmetical semigroups many different scenarios occur. It is all the more natural to pose the question, extending the original one of Littlewood \cite{Littlewood-JLMS}, what \emph{explicit, effective} conclusion\footnote{It was clear for long, and extends easily to the generality of the Beurling case, that once $\zeta(\rho_0)=0$, we must have $|\Delta(x)|\ge x^{\beta_0-\ve}$ for some $x$ values tending to $\infty$. However, this old result of Phragmen, see \cite{Ingham}, was completely ineffective, similarly to the later result of Schmidt \cite{Schmidt} stating that $\Delta(x)=\Omega(\sqrt{x})$, providing an improvement over the Phragmen result in case RH holds. These motivated Littlewood to ask for some effective oscillation result, explicit both in terms of the estimate and the localization of suitable $x$-values. Also note that in case we have RH, results of Littlewood \cite{Littlewood-CR}, as improved by Ingham \cite{Ingham-AA}, provided such an explicit result, but the interesting case of some "exceptional zero" (not on the critical line) could not be handled by them.} can be drawn for the oscillation of the error term $\Delta(x)$ from the existence of a given $\zeta$-zero? In fact, posing the problem Littlewood also pointed to the "interference difficulty" regarding the sum $\sum_\rho \frac{x^\rho}{\rho}$, appearing in the Riemann-von Mangoldt formula. The present paper addresses this question of Littlewood in the general context of Beurling number systems $\GG$.

A starting point to see what may be expected in this regard is the extension to the Beurling case of the classical formula of Riemann and von Mangoldt. For the Beurling case the formula was presented in Theorem 5.1 of \cite{Rev-MP}. This formula, in a slightly weakened form, says that for any $\ve>0$ we have
\begin{equation}\label{eq:RvM}
\Delta(x)=- \sum_{\rho \in \Z_{\ve,x}}
\frac{x^{\rho}}{\rho} + O_{\ve}\left(x^{\theta+\ve} \right),
\end{equation}
the sum running over $\ze$-zeroes of real part $\Re \rho\ge \theta+\ve$ and imaginary part $|\Im \rho|\le x$. Note that dropping the condition of finiteness and allowing $\Im \rho$ grow unbounded would in principle make the series divergent; this unpleasant divergence behavior makes this series representation of the error term $\Delta(x)$ hard to use. Nevertheless, the series suggests that once a zero $\rho_0$ is known, the sum has a term of the size $|x^{\rho_0}/\rho_0|=x^{\beta_0}/|\rho_0|$ (with $\beta_0:=\Re \rho_0$), and we can expect that the total sum--if cancelation of terms do not prevail for all values of large $x$--will be at least the same size, too.

In the classical case of the Riemann zeta function the problem of Littlewood was first answered by Turán \cite{Turan1}. Applying his celebrated power sum method \cite{Turan},\cite{SosTuran} Turán could prove an oscillation result essentially meaning $\Delta(x)=\Omega(x^{\beta_0-\ve})$ for all $\ve>0$ and with an effective, explicit lower bound and localization. The Turán result then was sharpened in several steps \cite{Stas0}, \cite{Knapowski} until Pintz \cite{Pintz1} reached $|\Delta(x)|\ge (1-\ve) x^{\beta_0}/|\rho_0|$, fully exploiting the presence of the term belonging to $\rho_0$. The results of Turán and others then were extended to various more general contexts, in particular to the case of prime ideals of algebraic number fields, see \cite{Stas-Wiertelak-1, Stas-Wiertelak-2, Stas1, Rev1}. Let us note that these effective results also furnished some localizations, where the large oscillations should occur, while related works \cite{PintzBasel, PintzNoordwijkerhout, Schlage-Puchta} produced various versions where the sharpness of the estimate was a little bit sacrificed in exchange for a sharper localizations, a trade-off so characteristic in these results.

At this stage, however, a new goal was set by Pintz: try to exploit \emph{both terms} belonging to $\rho_0$ and $\overline{\rho_0}$ together (the latter also occurring in view of the reflection principle), so that possibly the sum of these two terms could be extracted from \eqref{eq:RvM}. Interestingly, the first impression given by the formula \eqref{eq:RvM}, that is that even $(2-\ve) x^{\beta_0}/|\rho_0|$ should be reached, fails. The result below is the best what could be obtained in \cite{RevAA}.
\begin{theorem}\label{th:onezeroosci} Let $\ze(\rho_0)=0$
with $\rho_0=\beta_0+i\gamma_0$ be a zero of the Riemann zeta function.
Then for arbitrary $\ve>0$ we have for some suitable, arbitrarily large values
of $x$ the lower estimate $|\Delta(x)|\geq
(\pi/2-\ve) \frac{x^{\beta_0}}{|\rho_0|}$.
\end{theorem}
Of course, we do not know if RH holds, and if it holds, then a classical result \cite{Littlewood-CR} already says that $\Delta(x)=\Omega(\sqrt{x}\log\log\log x)$, larger than any individual term with some zero with $\beta_0=1/2$. But hypothetically, if there are zeroes off the critical line, then interference of same order terms can in principle extinguish some of the contribution of the above two terms, resulting in an oscillation of the size $(\pi/2-\ve) x^{\beta_0}/|\rho_0|$ only. To grasp this phenomenon the paper \cite{RevAA} considered a general class $Z$ of "zeta-type functions" and constructed an example in this class with some linearly dependent off-critical line zeroes and $|\Delta(x)|\le (\pi/2+\ve) x^{\beta_0}/|\rho_0|$. This, although we have no idea as to the validity of RH, but indicated that Theorem \ref{th:onezeroosci} is optimal in general (at least in the generality of the class $Z$).

The present work is part of a series. In \cite{Rev-MP} we proved a number of technical auxiliary results and estimates and concluded with the above mentioned Riemann-von Mangoldt formula \eqref{eq:RvM}. In the second part \cite{Rev-D} we worked out three theorems on the distribution of zeroes of the Beurling zeta function in the critical strip. With these we aimed to lay the ground for the extension to the Beurling case of the above results on the Littlewood question.

Here we will prove the following results.

\begin{theorem}\label{th:Bonezeroosci} Assume that $\GG$ satisfies Axiom A. Let $\ze(\rho_0)=0$ be a zero of the Beurling zeta function with $\rho_0=\beta_0+i\gamma_0$ satisfying $\beta_0>\theta$.

Then for arbitrary $\ve>0$ we have for some suitable, arbitrarily large values of $x$ the lower estimate $|\Delta(x)|\geq (1-\ve) \frac{x^{\beta_0}}{|\rho_0|}$;
moreover, in case $\gamma_0\ne 0$, even $|\Delta(x)|\geq (\pi/2-\ve) \frac{x^{\beta_0}}{|\rho_0|}$.
\end{theorem}

\begin{theorem}\label{th:Binterference} Let $\ve>0$, $\beta_0\in (1/2,1)$ and $1/2<r<\beta_0$ be given.

Then there exists an arithmetical semigroup $\GG$ satisfying Axiom A with $\theta=r$ (but with no better value), such that there is a zeta-zero $\rho_0=\beta_0+i\gamma_0$ of the Beurling zeta function $\zp$ with $\Re \rho_0=\beta_0$, and $|\Delta(x)|\leq (\pi/2+\ve) \frac{x^{\beta_0}}{|\rho_0|}$ holds true for all sufficiently large values of $x$.
\end{theorem}

Theorem \ref{th:Bonezeroosci} is very close to Theorem 1 of \cite{RevAA}, but there we did not think of the general Beurling situation. The zeta-type function class $Z$ appearing there was featured to cover algebraic number fields, but the special conditions we formulated are somewhat different than our conditions here. E.g. here we assume the "Average Ramanujan Condition" G--which is needed not directly in this part but in part two \cite{Rev-D} of the series to derive effective zero density estimates e.g.--while in \cite{RevAA} the first condition I.) reads only as $\N(x)\le K_1x\log^{K_2} x$, which we will get easily also for our situation from Axiom A, see \eqref{eq:Psirough} below. On the other hand here we assume a meromorphic continuation of $\zeta$ only to $\sigma>\theta$, while in \cite{RevAA} the same is assumed in Condition II.) for $\sigma>0$. So, the current setup is slightly different in several of its features from that of \cite{RevAA}.

That explains the difficulty in getting Theorem \ref{th:Binterference}, too. Here we need to obtain not only some nonnegative measure, the Mellin transfrom of which would be our constructed zeta function $\zeta$, but we need to make $\N(x)$ integer-valued and generated by a prime number system $\PP$. This can only be done by a proper random construction, based on \cite{DMV} and the improved version of \cite{BrouckeDebruyneVindas} and \cite{BrouckeVindas}.

In fact, along our way to this construction we will also prove an intermediate result, c.f. Theorem \ref{thm:Beurlingconstruction}, technical here but is likely to bear some interest on its own, which essentially says that for any prescribed finite system $\SSS$ of would-be zeta-zeros, one can construct a corresponding system of Beurling primes $\PP$ with $\Dp(x)\approx x-\sum_{\rho \in \SSS} x^\rho/\rho$ and $\N(x)$ satisfying Axiom A.

\section{Some auxiliary lemmas}\label{sec:aux}

\begin{lemma}\label{l::integralformula} For $a>0$, $b\in\CC$ and $c\in\RR$, we have
\begin{equation}\label{eq:integralformula}
\frac{1}{2\pi i} \int_{c-i\infty}^{c+i\infty} e^{as^2+bs} ds =
\frac{1}{2\sqrt{\pi a}} \exp\left(-\frac{b^2}{4a}\right).
\end{equation}
\end{lemma}

\begin{lemma}\label{l:estimates} The following estimates hold true.
\begin{itemize}
\item{(i)} For any $B\geq 1/2$,
$$ \int_B^{\infty} e^{-x^2} dx < e^{-B^2}.$$
\item{(ii)} For any $\lambda \geq 1$, $0<\alpha <1$ and $x\geq 1$ we have
$$ \log^{\lambda } x \leq e^{\lambda /\alpha+\lambda ^2} x^{\alpha}.$$
\item{(iii)} For any $P>0$ and $R\in\RR$ we have
$$\int_{-\infty}^{\infty} |\cos (Py+R)|\, e^{-y^2} dy  \leq
\frac{2}{\sqrt{\pi}}+\frac{2\pi}{P}.
$$
\end{itemize}
\end{lemma}

\begin{lemma}[{\bf Modified Cassels' power-sum
theorem}]\label{l:Cassels} If $1 \leq k$, $n\in \NN$, and
$w_1=\dots=w_k=1$, $w_{k+2j}=\overline{w_{k+2j-1}}$
($j=1,\dots,n$) are $k+2n$ complex numbers with
$w_{\ell}=r_{\ell} e^{i\alpha_{\ell}}$, ($|\alpha_{\ell}|\leq
\pi,\, \ell=1,\dots,k+2n$) then for any $H>0$ we have
$$
\max_{H\leq L \leq (2n+1)H} \Re \left( \sum_{\ell=1}^{k+2n} r_{\ell}^L
e^{i\alpha_{\ell} L} \right) \geq k.
$$
\end{lemma}

For the proofs of the above see \cite{RevAA}.


\section{Auxiliary results on the Beurling $\zeta$ function}\label{sec:basics}

In the following, we list a number of basic estimates and technical lemmas on the behavior of the Beurling $\zeta$ function.
Most of them are well-known, see, e.g., \cite{Knopf} or \cite{Beke} or \cite{DZ-16}. In \cite{Rev-MP} we elaborated on their proofs only for the explicit handling of the arising constants in these estimates. However, the Riemann-von Mangoldt formula in Proposition \ref{l:vonMangoldt} and the Carlson type density estimate Theorem \ref{th:density} were first given in \cite{Rev-MP} and \cite{Rev-D}, respectively. Here for the reader's convenience we recall those which we need here; for their proofs see \cite{Rev-MP} and \cite{Rev-D}. In this regard, however, we need to mention a few slight corrections, too, which change the values of the constants compared to \cite{Rev-MP}; for more explanations on the corrections see \cite{Rev-D}, where we have described them in more detail.

\subsection{Estimates for the number of zeros of  $\zeta$}\label{sec:zeros}

\begin{lemma}\label{l:Littlewood} Let $\theta<b<1$ and consider
any height $T\geq 5
$ together with the rectangle $Q:=Q(b,T):=\{
z\in\CC~:~ \Re z\in [b,1],~\Im z\in (-T,T)\}$. Then the number of
zeta-zeros $N(b,T)$ in the rectangle $Q$ satisfy
\begin{equation}\label{zerosinth-corr}
N(b,T)\le \frac{1}{b-\theta}
\left\{\frac{1}{2} T \log T + \left(2 \log(A+\kappa) + \log\frac{1}{b-\theta} + 3 \right)T\right\}.
\end{equation}
\end{lemma}

\begin{lemma}\label{l:zerosinrange}
Let $\theta<b<1$ and consider any heights $T>R\geq 5 $ together
with the rectangle $Q:=Q(b,R,T):=\{ z\in\CC~:~ \Re z\in [b,1],~\Im
z\in (R,T)\}$.

Then the number of zeta-zeros $N(b,R,T)$ in the rectangle $Q$ satisfies
\begin{equation}\label{zerosbetween}
N(b,R,T) \leq\frac{1}{b-\theta} \left\{ \frac{4}{3\pi} (T-R) \left(\log\left(\frac{11.4 (A+\kappa)^2}{b-\theta}\right)T\right)  + \frac{16}{3}  \log\left(\frac{60 (A+\kappa)^2}{b-\theta}\right)T\right\}.
\end{equation}

In particular, for the zeroes between $T-1$ and $T+1$ we have for
$T\geq 6$
\begin{align}\label{zerosbetweenone}
N(b,T-1,T+1) \leq \frac{1}{(b-\theta)} \left\{ 6.2 \log T +
6.2 \log\left( \frac{(A+\kappa)^2}{b-\theta}\right) + 24 \right\}.
\end{align}
\end{lemma}

\subsection{The logarithmic derivative of the Beurling $\zeta$}
\label{sec:logder}

\begin{lemma}\label{l:borcar} Let $z=a+it_0$ with $|t_0| \geq e^{5/4}+\sqrt{3}=5.222\ldots$ and $\theta<a\leq 1$. With $\delta:=(a-\theta)/3$ denote by $S$ the
(multi)set of the $\ze$-zeroes (listed according to multiplicity)
not farther from $z$ than $\delta$. Then we have
\begin{align}\label{zlogprime}
\left|\frac{\ze'}{\ze}(z)-\sum_{\rho\in S} \frac{1}{z-\rho}
\right| & < \frac{9(1-\theta)}{(a-\theta)^2}
\left(22.5+14\log(A+\kappa)+14\log \frac{1}{a-\theta} + 5\log |t_0|\right).
\end{align}

Furthermore, for $0 \le |t_0| \le 5.23$ an analogous estimate (without any term containing $\log |t_0|$) holds true:
\begin{equation}\label{zlogprime-tsmall}
\left|\frac{\ze'}{\ze}(z)+\frac{1}{z-1}-\sum_{\rho\in S} \frac{1}{z-\rho}
\right|  \le \frac{9(1-\theta)}{(a-\theta)^2}
\left(34+14\log(A+\kappa)+18\log \frac{1}{a-\theta}\right).
\end{equation}
\end{lemma}

\begin{lemma}\label{l:path-translates} For any given parameter
$\theta<b<1$, and for any finite and symmetric to zero set $\A\subset[-iB,iB]$ of cardinality $\#\A=n$, there exists a broken line $\Gamma=\Gamma_b^{\A}$, symmetric to the real axis and consisting of horizontal and vertical line segments only, so that its upper half is
$$
\Gamma_{+}= \bigcup_{k=1}^{\infty}
\{[\sigma_{k-1}+it_{k-1},\sigma_{k-1}+it_{k}] \cup
[\sigma_{k-1}+it_{k},\sigma_{k}+it_{k}]\},
$$
with $\sigma_j\in [\frac{b+\theta}{2},b]$, ($j\in\NN$),  $t_0=0$, $t_1\in[4,5]$ and $t_j\in
[t_{j-1}+1,t_{j-1}+2]$ $(j\geq 2)$ and satisfying that the
distance of any $\A$-translate $\rho+ i\alpha ~(i\alpha\in\A)$ of a $\zeta$-zero $\rho$ from any point $s=t+i\sigma \in \Gamma$ is at least $d:=d(t):=d(b,\theta,n,B;t)$ with
\begin{equation}\label{ddist-corr}
d(t):=\frac{(b-\theta)^2}{4n \left(12 \log(|t|+B+5) + 51 \log (A+\kappa) + 31 \log\frac{1}{b-\theta}+ 113\right)}.
\end{equation}
Moreover, the same separation from translates of $\zeta$-zeros holds also for the
whole horizontal line segments $H_k:=[\frac{b+\theta}{2}+it_k,2+it_k]$, $k=1,\dots,\infty$, and their reflections $\overline{H_k}:=[\frac{b+\theta}{2}-it_k,2-it_k]$, $k=1,\dots,\infty$, and furthermore the same separation holds from the translated singularity points $1+i\al$ of $\zeta$, too.
\end{lemma}

\begin{lemma}\label{l:zzpongamma-c} For any $0<\theta<b<1$ and symmetric to $\RR$ translation set $\A\subset [-iB,iB]$, on the broken line $\Gamma=\Gamma_b^{\A}$, constructed in the above Lemma \ref{l:path-translates}, as well as on the horizontal line segments $H_k:=[a+it_k,2+it_k]$ and  $\overline{H_k}$, $k=1,\dots,\infty$ with $a:=\frac{b+\theta}{2}$, we have uniformly for all $\alpha \in \A$
\begin{equation}\label{linezest-c}
\left| \frac{\ze'}{\ze}(s+i\alpha) \right| \le n
\frac{1-\theta}{(b-\theta)^{3}} \left(10 \log(|t|+B+5)+60\log(A+\kappa) + 42 \log\frac1{b-\theta}+ 140\right)^2.
\end{equation}
\end{lemma}

\subsection{A Riemann-von Mangoldt type formula of prime distribution
with zeroes of the Beurling $\zeta$}\label{sec:sumrho}

We denote the set of $\zeta$-zeroes, lying to the right of $\Gamma$, by $\Z(\Gamma)$, and denote $\Z(\Gamma,T)$ the set of those zeroes $\rho=\beta+i\gamma\in \Z(\Gamma)$ which satisfy $|\gamma|\leq T$. The next statement is Theorem 5.1 from \cite{Rev-MP}.

\begin{proposition}[Riemann--von Mangoldt formula]\label{l:vonMangoldt}
Let $\theta<b<1$ and $\Gamma=\Gamma_b^{\{0\}}$ be the curve defined in Lemma \ref{l:path-translates} for the one-element set $\A:=\{0\}$ with $t_k$ denoting the corresponding set of abscissae in the construction.
Then for any $k=1,2,\ldots$ (and whence $t_k\ge 4$) we have
$$
\psi(x)=x - \sum_{\rho \in \Z(\Gamma,t_k)}
\frac{x^{\rho}}{\rho} + O\left( \frac{1-\theta}{(b-\theta)^{3}}
\left(A+\kappa+\log \frac{x+t_k}{b-\theta}\right)^3 \left( \frac{x}{t_k} + x^b\right) \right).
$$
\end{proposition}

\subsection{A density theorem for $\ze$-zeros close to the $1$-line}
\label{sec:density}

In \cite{Rev-D} we used two additional assumptions to prove a density theorem on the zeroes of the Beurling zeta function. One is that the norm would actually map to the natural integers.
\begin{definition}[Condition B]\label{condB} We say that
\emph{Condition B} is satisfied, if $|\cdot|:\G\to\NN$, that is,
the norm $|g|$ of any element $g\in\G$ is a natural number.
\end{definition}

As is natural, we will write $\nu\in|\G|$ if there exists $g\in\G$
with $|g|=\nu$. Under Condition B we can introduce the arithmetical
function $G(\nu):=\sum_{g\in\G,~|g|=\nu} 1$, which is then an
arithmetical function on $\NN$. The next condition is a kind of
"average Ramanujan condition" for the Beurling zeta function.

\begin{definition}[Condition G]\label{condG} We say that
\emph{Condition G} is satisfied, if with a certain $p>1$ we have
for the function
\begin{equation}\label{Fpdef}
F_p(X):=\frac1X\sum_{g\in\G ; |g|\leq X} G(|g|)^p = \frac1X
\sum_{\nu \in|\G| ; \nu \leq X} G(\nu)^{1+p}=\frac1X \int_1^X
G^{p}(x)d\N(x)
\end{equation}
the property that
\begin{equation}\label{Gpcondi}
\log F_p(X) = o(\log X) \qquad (X\to \infty),
\end{equation}
that is, for any fixed $\varepsilon >0$
$F_p(X)=O(X^{\varepsilon})$.
\end{definition}

Note that in case $\log G(\nu) = o(\log \nu)$, i.e. when for all
$\ve>0$ we have $G(\nu)=O(\nu^{\ve})$, then Condition G is
automatically satisfied. Even this stronger $O(\nu^{\ve})$ order
estimate is proved for many important cases, see e.g. 2.4. Theorem
and 2.5. Corollary of \cite{Knopf}.

There are many natural examples of the above condition. For a discussion see the original book of Knopfmacher or Section 3.1 of \cite{Rev-D}. The main result of \cite{Rev-D} was the following Carlson-type density theorem, proved following the methods in \cite{Pintz9} and \cite{PintzNewDens}.
\begin{theorem}\label{th:density} Assume that $\G$ satisfies besides
Axiom A also Conditions B and G, too. Then for any $\varepsilon>0$
there exists a constant $C=C(\varepsilon,\G)$ such that for all
$\alpha>(1+\theta)/2$ we have
\begin{equation}\label{density}
N(\alpha,T)\leq C T^{\frac{6-2\theta}{1-\theta}(1-\alpha)+\ve}.
\end{equation}
\end{theorem}

Note that according to Lemma \ref{l:Littlewood} the above theorem gives a nontrivial--i.e., better than $O(T^{1+\ve})$--result only for $\alpha>\frac{5-\theta}{6-2\theta}$.

\section{Oscillation of $\Delta(x)$ "caused by a given zero" of
$\zeta$, Part I \\ The weighted average and its conditional upper estimate}\label{sec:upperesti}

Theorem \ref{th:Bonezeroosci} consists of two parts, but the part with $\gamma_0=0$ is easier and its proof can be easily
derived by adapting (and simplifying, where appropriate) the proof for the second, slightly more involved statement with $\gamma_0>0$. Therefore, we will present in detail only the proof of this more intriguing  part.

To obtain the $\gamma_0>0$ part of the assertion of Theorem \ref{th:Bonezeroosci}, actually we will
prove the slightly more precise statement below. Here and everywhere in the discussion $A'$, $A_1, A_2,\ldots$ stand for explicit constants depending only on the parameters $A,\kappa,\theta$ from Axiom A above.

\begin{theorem}\label{th:Bonezeroosci-plus} Let $\ze(\rho_0)=0$ with
$\rho_0=\beta_0+i\gamma_0$ and $\beta_0>\theta, ~ \gamma_0>0$.
Then for arbitrary $0<\ve<0.1$ and
\begin{equation}\label{eq:Ylimits}
\log Y > \max \left\{ \frac{5\log\frac{1}{\beta_0-\theta}}{\beta_0-\theta}, \frac{\log(8/\ve)}{\beta_0-\theta}, \frac{40}{\ve^2 \gamma_0^4}, \log|\rho_0|, A_9 \right\},
\end{equation}
there exists an $x$ in the interval
\begin{equation}\label{eq:xinterval}
I:=\left[Y,Y^{A_{10}\frac{\log(\gamma_0+5)}{(\beta_0-\theta)^2}}\right],
\end{equation}
such that
\begin{equation}\label{eq:oscillationpihalf}
\left| \D (x) \right| > \left(\frac{\pi}{2}-\varepsilon\right)
\frac{x^{\beta_0}}{|\rho_0|}.
\end{equation}
\end{theorem}

{\bf A preview of the proof.} We follow the proof of Theorem 1 in \cite{RevAA},
adapted to our setup and making use the analysis of the Beurling zeta function
worked out in the first and second parts of the series \cite{Rev-MP}, \cite{Rev-D}.

The main idea - borrowed from \cite{Pintz1} - is the use of a certain
weighted integral, which is evaluated in two ways. In one, we
estimate the total value by assuming an upper bound of $\D(x)$ in
$I$. We will also compute that a subinterval $[q,Q]$ supports almost the whole weight.
In the other evaluation we use contour integration and apply
Tur\'an's power sum theory for the evaluation -- i.e. for the
occasional lower estimation -- of the sum of residues. In this, the above
Modified Cassel's Power Sum Theorem (Lemma \ref{l:Cassels}) will be of importance.

Comparing the two estimates of the above mentioned weighted integral mean will finally provide the lower estimation of $\Delta(x)$.

\begin{proof} First we fix a few parameters as follows.
\begin{align}\label{eq:parameters}
m\geq \log Y, & \qquad M := 16 m, ~~\mu:=12 m,
\qquad q:=e^{M-\mu}=e^{4m}, \qquad Q:=e^{M+\mu}=e^{28m},  \notag
\end{align}
where $m$ is a continuous variable left to be chosen: we will do it so that
\begin{equation}
\label{eq:bBI} [q,Q] \subset I.
\end{equation}
We denote
\begin{equation}\label{eq:Kdef}
K:=\sup_{x\in I} \frac{|\D(x)|}{x^{\beta_0}}.
\end{equation}
We also introduce, as the Dirichlet-Mellin transform of $\D(x)$,
the in $\Re s>\theta$ meromorphic function
\begin{equation}\label{eq:Rdef}
D(s):=-\frac{\zeta'}{\zeta}(s) - \frac{s}{s-1} = \int_1^\infty x^{-s} d\Delta(x)=-
\int_{1}^{\infty} \D(x)dx^{-s} = s \int_{1}^{\infty} \D(x)x^{-s-1}
dx.
\end{equation}
Finally, for any complex parameter $w:=u+iv$ with $\beta_0 \leq u \leq
1$, $v>0$ we write
\begin{align}\label{eq:U}
U:=U(w)& := \frac{1}{2\pi i} \int_{(2)} D(s+w) e^{ms^2+Ms}ds
 \\ &= \frac{1}{2\pi i} \int_{(2)} \left( -\int_1^{\infty}
\D(x) \frac{d}{dx} (x^{-s-w}) dx \right) e^{ms^2+Ms} ds \notag
\\&= - \int_1^{\infty} \D(x) \frac{d}{dx} \left\{ x^{-w}
\frac{1}{2\pi i} \int_{(2)} e^{ms^2+(M-\log x)s}ds\right\}dx
\notag \\ &= - \int_1^{\infty} \D(x) \frac{d}{dx} \left\{ x^{-w}
\frac{1}{2\sqrt{\pi m}} \exp \left( -\frac{(\log x -M)^2}{4m}
\right)\right\} dx \notag \\ &= \frac{1}{2\sqrt{\pi m}}
\int_1^{\infty} \frac{\D(x) }{x}  x^{-w} \left\{\frac{\log x
-M}{2m} +w\right\} \exp \left( -\frac{(\log x -M)^2}{4m} \right)
dx, \notag
\end{align}
where the order of the integrations and the derivation were
changed and the integral formula of Lemma \ref{l::integralformula}
was applied.

Now let us split the integral for $U$ to three parts as
\begin{equation}\label{eq:Usplit}
U_1:=\int_1^q,~U_2:=\int_q^Q,~U_3:=\int_Q^{\infty}.
\end{equation}
For the general estimation of $\psi (x)$ and $\Delta(x)$, we may settle with the obvious estimates
\begin{align}\label{eq:Psirough}
(0\le ) ~\psi(x) &\le \sum_{g\in \G, |g|\le x} \log x = \N(x) \log x \le (A+\kappa) x \log x \qquad &(x\ge 1), \notag
\\ |\Delta(x)|&\le A' x \log x +1 \qquad \textrm{where} \quad A':=\max(1,A+\kappa) \qquad &(x\ge 1).
\end{align}

Suppose that
\begin{equation}\label{eq:wm}
|w|\leq \exp(4m).
\end{equation}
Assuming also $m\ge 4$ and $m \ge A':=\max(1,A+\kappa)$ we infer the estimate
\begin{align}\label{eq:U1}
|U_1|&\leq \frac{1}{2\sqrt{\pi m}} \int_1^{q} (A'\log x +
1) x^{-u} \left\{\frac{M-\log x}{2m} +|w|\right\} \exp \left(
-\frac{(\log x -M)^2}{4m} \right) dx \notag \\ & \leq
\frac{(A'4m+1)(\exp(4m)+8)}{2\sqrt{\pi m}} \int_1^{q} \exp
\left( -\left(\frac{(\log x -M)}{2\sqrt{m}}\right)^2-u \log x
\right) dx \notag \\ & < \frac{(5A'm)2\exp(4m)}{\sqrt{\pi}}  \int_{-8\sqrt{m}}^{-6\sqrt{m}}
e^{-y^2 -u(2\sqrt{m} y +M)} e^{2\sqrt{m} y +M} dy
\notag \\& < 9m^{2} e^{4m+M(1-u)+m(1-u)^2} \int_{6\sqrt{m}+(1-u)\sqrt{m}}^{\infty}
e^{-t^2} dt \notag \\ & < \exp\left(6m+16m(1 -u)-12m(1-u)-36m\right)
\leq e^{-26m}.
\end{align}
Here we have substituted $y:=(\log x-M)/(2\sqrt{m})$ and $t:=(1-u)\sqrt{m}-y$ and used Lemma \ref{l:estimates} (i),
$A'\le m$ and $3m<e^{m}$, valid for all $m\ge 4$.

Using that for $x>Q=e^{M+\mu}$ by \eqref{eq:wm} we have $\frac{\log x -M}{2m}+|w|<\log x+e^{4m}$ and taking into account Lemma \ref{l:estimates} (ii) we can estimate similarly the part $U_3$ as follows.
\begin{align}\label{eq:U3a}
|U_3|& \leq \frac{1}{2\sqrt{\pi m}} \int_Q^{\infty} (A'\log x +
1) x^{-u} \left\{\frac{\log x -M}{2m} +|w|\right\} \exp \left(
-\frac{(\log x -M)^2}{4m} \right) dx \notag \\
& \leq \frac{1}{2\sqrt{\pi m}} \int_Q^{\infty}\frac{2A'\log x (\log x+ e^{4m})}{x^u} \exp
\left( -\left(\frac{(\log x -M)}{2\sqrt{m}}\right)^2 \right) dx \notag
\\ & \leq \frac{A'}{\sqrt{\pi m}} \int_Q^{\infty} \left(e^{2/u+4}+e^{4m+1/u+1} \right)\exp
\left( -\left(\frac{(\log x -M)}{2\sqrt{m}}\right)^2 \right) dx.
\end{align}
Recall that $1/u\le 1/(\beta_0-\theta) \le \frac{\log Y}{\log(8/\ve)} \le \frac{\log Y}{\log 80} \le \frac14 \log Y \le \frac14 \log q = m$ by condition, hence $1/u\le m$ and $2/u\le 2m$, so that $e^{2/u+4}+e^{4m+1/u+1}\le e^{2m+4}+e^{5m+1} \le e^{3m}+e\cdot e^{5m} \le 3e^{5m}$, say. Therefore, after a change of variables an application of Lemma \ref{l:estimates} (iii) furnishes \begin{align}\label{eq:U3b}
|U_3| & < \frac{3 A'}{\sqrt{\pi m}} e^{5m} \int_{Q}^{\infty} \exp
\left( -\left(\frac{(\log x -M)}{2\sqrt{m}}\right)^2 \right) dx
\\ &=\frac{6A'}{\sqrt{\pi}} e^{5m} \int_{6\sqrt{m}}^{\infty}  \exp(-y^2) ~e^{2\sqrt{m}y+M}~dy \notag
\\ &= \frac{6A'}{\sqrt{\pi}} e^{5m+M+m} \int_{6\sqrt{m}}^\infty e^{-(y-\sqrt{m})^2} dy < 4 m e^{22m} e^{-25m} < e^{-2m}.\notag
\end{align}
Here in the last line we took into account that $\frac{6}{\sqrt{\pi}} A'\le 6m \le e^{m}$ for $m \ge 4$.

Next we define, combining the respective terms with $w=\rho_0$ and $w=\overline{\rho_0}$
\begin{equation}\label{Sdef}
S:=S(\rho_0):=U(\rho_0)+U(\overline{\rho_0}),
\end{equation}
and split it up the same way as we did for $U$ in \eqref{eq:Usplit}.
According to the above we then have
\begin{equation}\label{S13esti}
|S_1|+|S_3| \le 4e^{-2m}  \qquad \text{whenever} \quad m = \frac14 \log q \ge \max\left(\frac14 \log|\rho_0|, (A+\kappa), 4\right).
\end{equation}
Now let us consider the main part
$$
S_2=\frac{1}{2\sqrt{\pi m}} \int_q^Q \frac{\Delta(x)}{x} \left\{x^{-{\rho_0}}\rho_0 + x^{-{\overline{\rho_0}}}\overline{\rho_0} +  \left(x^{-{\rho_0}}+ x^{-{\overline{\rho_0}}}\right) \frac{M-\log x}{2m}\right\} \exp \left(
-\frac{(\log x -M)^2}{4m} \right) dx.
$$
Put $\alpha_0:=\arg \rho_0$ and $\alpha_1:=M\gamma_0-\alpha_0$. We can estimate $S_2$ using \eqref{eq:Kdef}, $[q,Q]\subset I$, the same as above substitution $y:=\frac{\log x - M}{2\sqrt{m}}$ and Lemma \ref{l:estimates} (iii) to get
\begin{align}\label{S2esti}
|S_2| & \leq
\frac{1}{\sqrt{\pi m}} \int_q^Q \frac{K}{x} \left\{\left|\rho_0\cos(\gamma_0\log x-\alpha_0)\right| + \frac{|M-\log x|}{2m} \right\} \exp \left(
-\left(\frac{\log x -M}{2\sqrt{m}} \right)^2\right) dx \notag
\\ & \leq
\frac{K |\rho_0|}{\sqrt{\pi m}} \int_q^Q \frac{|\cos(\gamma_0\log x-\alpha_0)|}{x}  \exp \left(-\left(\frac{\log x -M}{2\sqrt{m}} \right)^2\right) dx  + \frac{2K}{\sqrt{\pi m}} \int_{-6\sqrt{m}}^{6\sqrt{m}} |y| e^{-y^2} dy
\notag \\ & \leq \frac{2K |\rho_0|}{\sqrt{\pi}} \int_{-\infty}^{\infty} |\cos(2\sqrt{m}\gamma_0 y+\alpha_1)| \exp \left(-y^2\right) dy  + \frac{2K}{\sqrt{\pi m}} \int_{-\infty}^{\infty} |y| e^{-y^2} dy \notag
\\ & \leq \frac{2K |\rho_0|}{\sqrt{\pi}} \left(\frac{2}{\sqrt{\pi}}+\frac{2\pi}{2\sqrt{m}\gamma_0} \right) + \frac{2K}{\sqrt{\pi m}}
= \frac{4K |\rho_0|}{\pi} \left(1+\frac{\pi \sqrt{\pi} + \sqrt{\pi}\gamma_0/|\rho_0|}{2\sqrt{m}\gamma_0} \right) \notag
\\& \le \frac{4K |\rho_0|}{\pi} \left(1+ \frac{4}{\sqrt{m}\gamma_0} \right),
\end{align}
on noting that $(\pi+1)\sqrt{\pi}=7.340781848...<8$.
Finally, let us combine this with the estimates for $S_1$ and $S_3$: in all we are led to
\begin{equation}\label{S-upperesti}
|S| \le \frac{4K |\rho_0|}{\pi} \left(1+\frac{4}{\sqrt{m}\gamma_0} \right)+4e^{-2m},
\end{equation}
whenever the conditions in \eqref{S13esti} for $m$ hold true. 

This will be compared to the lower estimation of the next section.

\section{Oscillation of $\Delta(x)$ "caused by a given zero" of
$\zeta$, Part II \\ Lower estimate by contour integration and power sum theory}\label{sec:loweresti}


In the second part we calculate $S$ by using the first form of $U$ in \eqref{eq:U}. We transfer the line of integration of $U(w)$ with $w=u+iv$ from $(\sigma=2)$ to the contour $\Gamma-u$, where $\Gamma:=\Gamma^{\A}_{b}$ is provided by Lemma \ref{l:path-translates}, with $\A:=\{-v,0,v\}$ and $b\in (\theta,u)$ a parameter to be chosen later, while $w=u+iv$ will be chosen $\rho_0=\beta_0+i\gamma_0$, as above. The transition of the contour of integration can be done easily due to the estimates of Lemma \ref{l:zzpongamma-c} and the uniform bound $|e^{ms^2+Ms}|=O_m(e^{-t^2})$ holding uniformly in the strip $-1\le \sigma \le 2$ and $s=\sigma+it$. By an application of the Residue Theorem we thus find after the change of the integration path
\begin{align}\label{Uotherway}
U(w) & = \frac{1}{2\pi i} \int_{\Gamma-u} D(s+w) e^{ms^2+Ms}ds + \sum^{\star}_{\rho} \exp\left(m(\rho-w)^2+M(\rho-w)\right),
\end{align}
where the $\star$ indicates that exactly those zeroes of the Beurling zeta function are taken into account (and then according to multiplicity) which lie to the right of the new contour $\Gamma-u$, more precisely, for which $\rho-w$ is to the right of $\Gamma-u$. Recall that the singularities of $D(s+w)$ are exactly at translates $\rho-w$ of zeroes $\rho$ of $\zeta$ with residues according to multiplicity; and that by construction all translated $\zeta$-zeroes $\rho-w$ avoid points $s-u=\sigma-u+it$ of the curve $\Gamma-u$ --that is, all vertically translated zeroes $\rho-iv$ avoid the points $s=\sigma+it$ of the curve $\Gamma$--by at least $d:=d(t):=d(b,\theta,n,v;t)$ given in \eqref{ddist-corr}. The analogous statement holds for $\overline{w}$ and $U(\overline{w})$, too.

Here the integral can be estimated by Lemma \ref{l:zzpongamma-c} taking into account $\theta<b<u$, $a:=\frac{b+\theta}{2}$ and the construction of $\Gamma_b^\A$ as follows.
\begin{align}\label{Newcontourint}
\bigg| \frac{1}{2\pi i} & \int_{\Gamma-u}  D(s+w) e^{ms^2+Ms}ds \bigg|
\\ & \le \frac{1}{2\pi} \int_{\Gamma-u} \frac{A_1}{(b-\theta)^{3}} \left(\log (|t|+v+5) + \log \frac1{b-\theta}\right)^2 \exp\left(m((u-a)^2-t^2) +M(b-u)\right) |ds| \notag
\\ & \le \frac{A_2}{(b-\theta)^5} e^{m(u-a)^2+M(b-u)} \int_{\Gamma-u} \log^2 (|t|+v+5) \exp(-mt^2) |ds|. \notag
\end{align}

By construction, the broken line $\Gamma$ consists of horizontal line segments $H_k$ of length $\le \frac12(b-\theta)$ at height $t_k$, and vertical segments the horizontal projection of which covers the imaginary axis exactly (apart from endpoints). Therefore,
$$
\int_{\Gamma-u} \log^2 (|t|+v+5) \exp(-mt^2) |ds| \le 2 \int_{0}^\infty \log^2 (t+v+5) e^{-mt^2} dt + (b-\theta) \sum_{k=1}^{\infty} \log^2 (t_k+v+5) e^{-mt_k^2}.
$$
Using the standard Vinogradov notation $\ll$ for explicit numerical constants only, for the integral here we easily see $\int_{0}^\infty = \int_{0}^{v+5}  +\int_{v+5}^\infty \le \log^2(2v+10)\int_0^\infty e^{-t^2}dt + \int_5^{\infty} \log^2(2t) e^{-t^2} dt \ll \log^2(v+5)$. Recalling that by construction $t_1\ge 4$ and $t_k\ge t_{k-1}+1$, ($k\ge 2$), we get a similar estimate for the sum. Therefore,
$$
\int_{\Gamma-u} \log^2 (|t|+v+5) \exp(-mt^2) |ds| \ll \log^2(v+5).
$$
Collecting the above estimates and putting $w=\rho_0, \overline{\rho_0}$ we are led to
\begin{equation}\label{Stranslatedintegral}
\bigg| \frac{1}{2\pi i} \int_{\Gamma-u}  \left\{D(s+\rho_0)+D(s+\overline{\rho_0})\right\} e^{ms^2+Ms}ds \bigg| \le \frac{A_3 \log^2(\gamma_0+5)}{(b-\theta)^5} e^{m(\beta_0-a)^2+M(b-\beta_0)}.
\end{equation}

Next we see to the estimations of the various parts of the right hand side sum of \eqref{Uotherway}.

Keeping the notation $a=\frac{b+\theta}{2}$ used in the construction of $\Gamma$ let us write
\begin{align}\label{Zfarrootsum}
Z_1(w)&:=\sum^{\star}_{\rho;~ |\Im \rho-v|\ge 5} \exp\left(m(\rho-w)^2+M(\rho-w)\right)  \notag
\\ \notag & \le \sum_{k=5}^\infty \exp\left(m(1-u)^2-mk^2+M(1-u)\right) \left\{N(a,v-k-1,v-k)+ N(a,v+k,v+k+1)\right\}
\\ \notag & \le e^{m(1-u)^2+M(1-u)} \sum_{k=5}^\infty e^{-k^2m} \frac{1}{a-\theta}\left(A_4+A_5\log(v+k)\right)
\\ & \le \frac{A_6}{b-\theta} \log(v+5) e^{m(1-u)^2+M(1-u)-25m} \le \frac{A_6}{b-\theta} \log(v+5) e^{-8m},
\end{align}
referring to Lemma \ref{l:zerosinrange} in the third line and then calculating similarly as we did above for the sum $\sum_{k=1}^{\infty} \log^2 (t_k+v+5) e^{-mt_k^2}$.

Applying this to $w=\rho_0, \overline{\rho_0}$ and combining with \eqref{Sdef}, \eqref{Uotherway} and \eqref{Stranslatedintegral}, we are led to
\begin{equation}\label{beforeCassels}
|S(\rho_0)|\ge \left|P \right| - \frac{A_7 \log^2(\gamma_0+5)}{(b-\theta)^5} e^{m(\beta_0-a)^2+M(b-\beta_0)},
\end{equation}
where $P:=P(\rho_0)$ is defined as
\begin{align*}
P& :=\sum^{\star}_{\rho;~ |\Im \rho-\gamma_0|< 5} \exp\left(m(\rho-\rho_0)^2+M(\rho-\rho_0)\right) + \sum^{\star}_{\rho;~ |\Im \rho+\gamma_0|< 5} \exp\left(m(\rho-\overline{\rho_0})^2+M(\rho-\overline{\rho_0})\right)
\\ &= \sum^{\star}_{\rho;~ |\Im \rho-\gamma_0|< 5} \exp\left(m(\rho-\rho_0)^2+M(\rho-\rho_0)\right) + \exp\left(m(\overline{\rho}-\overline{\rho_0})^2+M(\overline{\rho}-\overline{\rho_0})\right).
\end{align*}
To reach a concrete control over the arising error terms we now choose $b:=\frac{\beta_0+\theta}{2}$ (and $a=\frac{\beta_0+3\theta}{4}$ accordingly), and calculate $e^{m(\beta_0-a)^2+M(b-\beta_0)}=e^{m\frac{9}{16}(\beta_0-\theta)^2-8m(\beta_0-\theta)}<e^{-7m(\beta_0-\theta)}$. Therefore, \eqref{beforeCassels} and a little calculus yields
\begin{equation}\label{closetoCassels}
|S(\rho_0)|\ge \left|P \right| - e^{-3m(\beta_0-\theta)},\quad \textrm{if} \quad m\ge \max\left( \frac{5\log\frac{1}{\beta_0-\theta}}{\beta_0-\theta}, \frac{\log A_7}{\beta_0-\theta}, \frac{\log\log(\gamma_0+5)}{\beta_0-\theta} \right),
\end{equation}
say.

Now, $P$ can be written as a sum of pure powers (i.e. without coefficients), where the general term takes either the form
$$
\exp\left(m(\rho-\rho_0)^2+M(\rho-\rho_0)\right)=\exp\left(m\left((\rho-\rho_0)^2+16(\rho-\rho_0)\right)\right)=e^{m\lambda(\rho)},
$$
with $\lambda(\rho):=(\rho-\rho_0)^2+16(\rho-\rho_0)$, or exactly its conjugate
$$
\exp\left(m(\overline{\rho}-\overline{\rho_0})^2+M(\overline{\rho}-\overline{\rho_0})\right)=\exp\left(m\left((\overline{\rho}- \overline{\rho_0})^2+16(\overline{\rho}-\overline{\rho_0})\right)\right)=e^{m\overline{\lambda(\rho)}}.
$$
In the sum there are at most $2N(a,\gamma_0-5,\gamma_0+5)$ terms. Therefore, the number of terms is estimated for $\gamma_0\le 10$ by $2N(a,15) \le \frac{A_8}{\beta_0-\theta}$ according to Lemma \ref{l:Littlewood}, or for $\gamma_0>10$ by $\frac{A_8}{\beta_0-\theta}\log(\gamma_0+5)$ with reference to \eqref{zerosbetween} in Lemma \ref{l:zerosinrange}.

Out of the conjugate pairs of terms of $P$ there are at least one pair--hence at least two terms--which must be exactly 1. Therefore, Lemma \ref{l:Cassels} gives that in any interval of the form $J:=J(H):=[H,\frac{A_8\log(\gamma_0+5)}{\beta_0-\theta} H]$ there exists some $m$ for which $|P|\ge 2$. Taking into account \eqref{S-upperesti} and \eqref{closetoCassels} we therefore obtain
\begin{equation}\label{Salmostready}
\frac{4K |\rho_0|}{\pi} \left(1+\frac{4}{\sqrt{m}\gamma_0} \right)+4e^{-2m} \ge |S| \ge 2-e^{-3m(\beta_0-\theta)},
\end{equation}
for the particular value of $m \in J$, assuming that all the conditions appearing above in \eqref{eq:bBI}, \eqref{S13esti} and \eqref{closetoCassels} are met. Multiplying by $\frac{\pi}{4|\rho_0|}$ and writing in $m\ge H$, $\beta_0-\theta\le 1$, we obtain
$$
K \ge \frac{1}{|\rho_0|} \left(\frac{\pi}{2}-\frac{\pi}{\sqrt{H}\gamma_0 |\rho_0|} - \frac{5\pi}{4} e^{-2H(\beta_0-\theta)}\right).
$$
At last we choose $H:=\log Y$, so that $q:=e^{4m}\ge e^{4H} \ge Y$ and $m\le \frac{A_8}{\beta_0-\theta}\log(\gamma_0+5) \log Y$. Note that the condition \eqref{eq:bBI} will be met if $A_{10}\ge 28A_8$, ensuring also $Q=\exp(28 m) \le Y^{\frac{A_{10}}{\beta_0-\theta}\log(\gamma_0+5)}$ and whence the validity of the upper estimation $\sup_{x \in [q,Q]} \frac{|\Delta(x)|}{x^{\beta_0}}\le K:=\sup_{x \in I} \frac{|\Delta(x)|}{x^{\beta_0}}$ according to \eqref{eq:Kdef}. Further,
$$
\frac{\pi}{\sqrt{H}\gamma_0 |\rho_0|} \le \frac{\pi}{\sqrt{H}\gamma^2_0} \le \frac12 \ve \qquad \textrm{if} \quad H\ge \frac{40}{\ve^2 \gamma_0^4},
$$
and
$$
\frac{5\pi}{4} e^{-2H(\beta_0-\theta)} \le \frac12 \ve \qquad \textrm{if} \quad H\ge \frac{\log(8/\ve)}{2(\beta_0-\theta)}.
$$
It follows that $K\ge \frac{\pi/2-\ve}{|\rho_0|}$ whenever these conditions are all met. However, for $H:=\log Y$ the assumptions \eqref{eq:Ylimits} contain both assumptions here, further, they suffice for \eqref{S13esti} and \eqref{closetoCassels} to hold, so that the result is proved.
\end{proof}

\section{General remarks on the sharpness of Theorem \ref{th:Bonezeroosci} and a special sine polynomial}\label{sec:polynomial}

In the following our goal will be to show--by giving appropriate examples of Beurling prime number and integer systems--that Theorem \ref{th:Bonezeroosci-plus} is optimal, i.e, there exist systems which satisfy all assumptions, yet an oscillation of the size $(\pi/2+\ve)x^{\beta_0}/|\rho_0|$ fails. In other words, we seek systems where $|\Delta_{\G}(x)| \le (\pi/2+\ve)x^{\beta_0}/|\rho_0|~(x\ge x_0)$ with a certain $\zeta$-zero $\rho_0$ of $\zeta$.

It is easy to see that such a Beurling zeta function must have $\beta_0=\theta_0$ where
$$
\theta_0:=\max(\theta, \sup\{\Re \rho~:~ \zeta(\rho)=0\}).
$$
Indeed, if otherwise then with any other zero $\rho_1=\beta_1+\gamma_1$ with $\beta_1>\beta_0$ already Theorem \ref{th:Bonezeroosci-plus}, when applied to this new zero, provides essentially larger oscillation (of the order of $x^{\beta_1}$). In particular, it follows that $\theta_0<1$, as it is well-known that the line $\Re s=1$ does not contain a zero of the Beurling zeta function under much weaker hypothesis than Axiom A.

Moreover, the only reasonable choice is $\gamma_0=\min\{\gamma >0~:~\zeta(\beta_0+i\gamma)=0\}$: for picking other zeroes for $\rho_0$ would simply decrease the constant $(\pi/2+\ve)/|\rho_0|$, making our task more difficult, and, in view of Theorem \ref{th:Bonezeroosci-plus}, even impossible.

Given that we are talking about a Beurling zeta function, arising from a number system, which belongs to a real valued $\N(x)$, the Beurling zeta function is also real valued for real variables $\sigma \in \RR$. Hence by the reflection principle together with any zeta-zero $\rho$ also the conjugate zero $\overline{\rho}$ occurs. Now let us take a look at the terms, "caused by a given zeta zero $\rho$", as they occur in the Riemann-von Mangoldt type formula of Proposition \ref{l:vonMangoldt}. They provide
$$
\frac{x^\rho}{\rho}+\frac{x^{\overline{\rho}}}{\overline{\rho}}= 2 x^\beta \frac{\cos(\gamma\log x -\alpha)}{|\rho|}\qquad \left( \alpha:=\arg(\rho)=\arctan(\gamma/\beta) \right),
$$
or, if we have a series $\rho_k$ of known zeroes with $\beta_k=\beta_0(=\theta)$, then
$$
\sum_{k} \frac{x^{\rho_k}}{\rho_k}+\frac{x^{\overline{\rho_k}}}{\overline{\rho_k}}= 2 x^\beta \sum_{k} \frac{\cos(\gamma_k\log x -\alpha_k)}{|\rho_k|}\qquad \left(\alpha_k:=\arg(\rho_k)=\arctan(\gamma_k/\beta)\right).
$$
To handle these terms easier, let us assume that $\gamma_0$ is chosen very large; then $\alpha_k\approx \pi/2$ and $|\rho_k|\approx \gamma_k$, so that the above sum is approximately
$$
-2 x^\beta \sum_{k}  \frac{\sin(\gamma_k\log x)}{\gamma_k}.
$$
In principle, the sum here can contain infinitely many elements as well, but then the delicate issue of convergence arises. In any case, let us see what we may expect from such a sum. The well-known Fourier series \begin{equation}\label{eq:signseries}
\frac{\pi}{2} \sign (y) = \sum_{k=1}^\infty 2\frac{\sin((2k+1)y)}{2k+1}
\end{equation}
suggests that we should strive for getting $\Delta(x) \approx \frac{\pi}{2} x^{\beta_0} \sign(\gamma_0\log x)/\gamma_0$, that is $\psi(x)\approx x+\frac{\pi}{2} x^{\beta_0} \sign(\gamma_0\log x)/\gamma_0$.

There is only one obstacle here, but a serious one. As said, the two conjugate terms belonging to $\rho$ and its conjugate together rise to a size of $2x^{\beta_0}/|\rho_0|$ time to time. Therefore, to uniformly push down the oscillation, caused by them to only $(\pi/2) x^{\beta_0}/|\rho_0|$, we heavily rely on the interference of other terms of similar size. That happens in the slowly and non-uniformly convergent series of $(\pi/2)\sign (y)$, but to construct a $\Delta(x)$ of that same terms would require a sequence of $\zeta$-zeroes at $\rho_k=\beta_0+(2k+1)\gamma_0$, i.e. a constant times $T$ zeroes on the $\Re s=\theta_0=\beta_0$ line. However, that is impossible for $\beta_0 > \theta + \frac{5}{6}(1-\theta)$, as was recently demonstrated--upon the additional assumptions of Conditions B and G--by the new density result in Theorem \ref{th:density}. On such a line, and in general in a rectangular domain $[\alpha,2]\times[-iT,iT]$ with any $\alpha> \theta + \frac{5}{6}(1-\theta)$, only $o(T)$ zeros can occur. Then, similarly to the analysis in \cite{RevAA}, it even follows that we cannot get better uniform bounds for a properly rare sequence of zeros (which at least meet the criteria, posed by the density theorem), then for simply assuming only finitely many zeros. However, from \eqref{eq:signseries} there is not so easy to come to a finite sum of the same low maximum norm (so the same level of interference extinguishing a $(1-\pi/4)$ portion of the magnitude of the first term). This is known as Gibbs phenomenon or overshooting convergence: see, e.g., page 61 of \cite{Zyg}. According to this phenomenon, partial sums $S_n(\frac{\pi}{2} \sign)$ have definitely larger maximum norm than $\frac{\pi}{2} \sign$ itself. More precisely, $\lim_{n\to \infty} \|S_n(\frac{\pi}{2} \sign)\| = \int_0^\pi (\sin x/x) dx \approx 1.8519\ldots > \pi/2=\|\frac{\pi}{2}\sign\|$, although already less than 2, the maximal size of the first summand.

This is a point when construction of a finite sum $S$ of terms from the series \eqref{eq:signseries}, with about the same low maximum norm $\pi/2+\ve$ as the total sum itself, becomes of some challenge. That was first solved in \S 6 of \cite{RevAA} by a probabilistic construction. Later R\'egis de la Bret\`eche and G\'erald Tenenbaum furnished a deterministic construction, too, through so-called "entieres friables", see in particular the explanation following Th\'eor\`eme 2.2 of \cite{BT}. Both the original probabilistic construction and the later arithmetical construction (so-called P-summability) applies to a wide class of Fourier series, borrowing some importance to the otherwise seemingly rather special question here, see \cite{Rev-JAT} and \cite{BT-2}.

The existence of such a special sine polynomial will be one starting point for our construction, so that we formulate it here as a lemma.
\begin{lemma}\label{l:finitesum} For any $\ve>0$ there exists a natural number $N$ and a finite sequence $\{n_k\}_{k=1}^N \subset \NN$, such that with $n_0:=0$ the sine polynomial $S(y):=2\sum_{k=0}^N \frac{\sin((2n_k+1)y)}{2n_k+1}$ has maximum norm $\| S\|\le \pi/2+\ve$.
\end{lemma}

\section{A construction of a Beurling number system with given zeroes of $\zp$}\label{sec:zetawithgivenzeroes}

The above suggests that we will need a system of Beurling primes $\PP$ such that the corresponding Beurling zeta function will have a special configuration of zeroes while satisfy certain analytic and order estimate conditions as well. In this section we will present such a construction in a greater generality with possible further applications in mind.

So we set to the following task. Let $r$ be a real number parameter satisfying $1/2 \le r <1$. Also, let a finite multiset $\SSS$ of would-be zeta-zeroes be given. We assume that each element $\rho=\beta+i\gamma \in \SSS$ satisfies $\Re \rho=\beta \in (r,1)$, always, and is listed according to multiplicity, moreover, $\SSS$ is symmetric with respect to the real axis (so that $\rho$ is listed with the same multiplicity as $\overline{\rho}$, the condition being self-evident for real zeroes with $\gamma=0$). Then, the task is to construct a Beurling system of primes $\PP$ and corresponding set of Beurling generalized integers $\N$ subject to Axiom A, and such that $\zp(s)$ has zeta-zeroes in the halfplane $\Re s>r$--or even in $\Re s>1/2$--exactly as prescribed by $\SSS$. Moreover, we want, roughly speaking, that the system satisfy Axiom A with "the best value" of $\theta$ to be $r$. More precisely, $\zp$ be analytic in $\Re s > 1/2$ except for a simple pole at $s=1$, and possibly another one at $s=r$ if $r$ was exceeding $1/2$; and Axiom A is to be satisfied with $\theta$ as either the given value $r>1/2$ (and then, according to the assumed singularity at $s=r$, with no smaller value than $r$), if $r$ exceeded $1/2$, or with all $1/2+\ve$ with any $\ve>0$, if $r=1/2$.

In the course of our work we will establish a number of other useful properties, too, and at the end of the section we will summarize our findings in Theorem \ref{thm:Beurlingconstruction}. That will be combined with Section \ref{sec:polynomial} in the next section to prove Theorem \ref{th:Binterference}, too.

Our Beurling number system $\N$ will arise as a result of a prime selection procedure, designed to approximate a pre-set distribution function as well as possible. This part is far from trivial, and the method of doing so involves probabilistic considerations. The approach was introduced into the study of Beurling number systems by the breakthrough work of Diamond, Montgomery and Vorhauer \cite{DMV}, and then refined by Zhang \cite{Zhang7} as follows.
\begin{theorem}[Diamond-Montgomery-Vorhauer-Zhang]\label{th:DMVZ} Let $f:(1,\infty)\to \RR_{+}$ be a non-negative, locally integrable function with $\int_1^\infty f(u)du =\infty$, and satisfying the "pointwise Chebyshev bound" $f(x)\ll 1/\log x$. Write $F(x):=\int_1^x f(y)dy$.

Then there exists a set of generalized primes $\PP=\{p_j\}_{j=1}^\infty$ such that for all $x\ge 1$ it holds $|\pi_{\PP}(x)-F(x)|\le \sqrt{x}$, and, moreover, we have for all $x\ge 1$ and all $t\in\RR$ the estimate
\begin{equation}\label{pitestimateDMVZ}
\left| \sum_{p_j\le x} p_j^{-it} - \int_1^x u^{-it} f(u) d u \right| \ll \sqrt{x} +\sqrt{\frac{x\log(|t|+1)}{\log(x+1)}}.
\end{equation}
\end{theorem}
Recently a nice sharpening of the method appeared in \cite{BrouckeVindas}. This latter result--Theorem 1.2 in \cite{BrouckeVindas}--will not be indispensable for us, but we will take some slight advantage of it, too.
\begin{theorem}[Broucke-Vindas]\label{th:BV} Let $F$ be a non-decreasing right-continuous function tending to $\infty$, with $F(1)=0$ and satisfying the "global Chebyshev bound" $F(x)\ll x/\log x$.

Then there exists a set of generalized primes $\PP=\{p_j\}_{j=1}^\infty$ such that for all $x\ge 1$ it holds $|\pi_{\PP}(x)-F(x)|\le 2$, and, moreover, we have for all $x\ge 1$ and all $t\in\RR$ the estimate
\begin{equation}\label{pitestimate}
\left| \sum_{p_j\le x} p_j^{-it} - \int_1^x u^{-it} d F(u) \right| \ll \sqrt{x} +\sqrt{\frac{x\log(|t|+1)}{\log(x+1)}}.
\end{equation}
\end{theorem}
We will exploit the full strength of the above marvelous results. Moreover, our proof of Theorem \ref{thm:Beurlingconstruction} will draw much from the proof of Theorem 3.1 of \cite{BrouckeVindas}.

Denote $B:=\max\{\beta=\Re \rho ~:~ \rho\in\SSS\}$ and $N:=\#\SSS$ (counted according to multiplicity). Take a further constant $M\in \NN$ and consider the function
\begin{equation}\label{eq:f}
f_0(x):=x + \frac1{r}x^r+2M\sqrt{x}-\sum_{\rho \in \SSS} \frac{x^{\rho}}{\rho} \qquad  (x\ge 1),
\end{equation}
where here and everywhere else we will mean summation--or taking product--over $\SSS$ according to multiplicity. Direct differentiation yields
$$
f_0'(x)=1+x^{r-1}+Mx^{-1/2}-\sum_{\rho\in\SSS} x^{\rho-1} \ge 1+Mx^{-1/2}-Nx^{B-1}.
$$
The right hand side will be nonnegative if the constant $M$ has been chosen large enough. A little calculus gives that $M\ge M_0:=2N^{\frac{1}{2(1-B)}}$ suffices. It is clear, too, that $f_0(x)=O(x)$, and $f_0(x) \sim x$ as $x\to \infty$. Therefore, the function
\begin{equation}\label{Fdef}
F(x):=\int_1^x df(y) \quad \textrm{where} \quad f'(y):=\frac{1-1/y}{\log y} f'_0(y), \quad \textrm{i.e.} \quad df(y):=\frac{1-1/y}{\log y} df_0(y)
\end{equation}
is a well-defined, continuously differentiable, nondecreasing function, and it admits a Chebyshev bound $F(x)\ll x/\log x$, too, whence it is subject to all the requirements of Theorem \ref{th:BV} of Broucke and Vindas above. As a result, there exists a prime number system $\PP$ with $|\pi_{\PP}(x)-F(x)|\le 2$ and satisfying \eqref{pitestimate}, too. Referring to Theorem \ref{th:DMVZ} only would give here $|\pi_{\PP}(x)-F(x)|\ll \sqrt{x}$, still well sufficient for us, as will be seen below.

For this prime number system $\PP$ the corresponding $\vartheta$ function is defined as $\vartheta(x):=\int_1^x \log y d\pi_{\PP}(y)$. Let us see that it will satisfy $\vartheta(x)=f_0(x)+O(\log x)$. Indeed,
\begin{align*}
\vartheta(x) &:= \left[\log y \pi_{\PP}(y) \right]_1^x - \int_1^x \frac{\pi_{\PP}(y)}{y}dy =\log x ~\pi_{\PP}(x)- \int_1^x \frac{F(y)}{y}dy
+O(\log x)
\\
&= \log x(F(x)+O(1)) -\left\{\int_1^x \left(\int_1^y \frac{1-1/z}{\log z} df_0(z) \right) \frac{dy}{y}\right\} + O(\log x)
\\ &= F(x)\log x - \left\{\int_1^x \left(\int_z^x \frac{dy}{y}\right) \frac{1-1/z}{\log z}  df_0(z) \right\}+ O(\log x)
\\ &= F(x)\log x - \left\{\int_1^x (\log x- \log z) \frac{1-1/z}{\log z}  df_0(z) \right\}+ O(\log x)
\\ &= F(x)\log x - \left\{\log x F(x) - f_0(x)+\int_1^x  \frac{df_0(z)}{z}\right\}+ O(\log x) =f_0(x) + O(\log x).
\end{align*}

As a direct consequence, for the respective von Mangoldt summatory function $\psi_{\PP}(x)=\sum_{n=1}^{[\log x]} \vartheta(x^{1/n})$ we necessarily have $\psi_{\PP}(x)=f_0(x)+g(x)$ with $g(x)= O(\theta(\sqrt{x}))=O(\sqrt{x})$. Using Theorem \ref{th:DMVZ}, the same argument would furnish the weaker $\vartheta(x)=x+O(\sqrt{x} \log x)$ only--due to the change of $O(1)$ to $O(\sqrt{x})$ in the second line--but afterwards for $\psip(x)$ we are to get the only slightly weaker result $g(x)=O(\sqrt{x}\log x)$. However, for this the dominant error comes from the error of the applied theorem, whence cannot be "tricked out" by a modification of the prescribed distribution for $\vartheta$, i.e., for $\pip$, while using the strong $O(1)$ result of Broucke and Vindas, we can even achieve $g(x)=O(x^\varepsilon)$ for $\psip$. For more about this exploitation of the full strength of Theorem \ref{th:BV} see Remark \ref{rem:trick} below.

\medskip
Next, we are to compute the respective Beurling zeta function $\zeta_{\PP}$ from the Mellin transform of $\psi_{\PP}$. Recall that in our terminology the Mellin transform is defined as $\MM(\phi)(s):=\int_1^\infty x^{-s} d\phi(x)=-\phi(1)-\int_1^\infty \phi(x) dx^{-s} $, and the basic connection between the Beurling zeta function $\zeta_{\PP}$ and the von Mangoldt summatory function $\psi_\PP$ is that $\MM(\psi_{\PP})(s)=-\frac{\zeta'_{\PP}}{\zeta_{\PP}}(s)$.

Denote the power function $x\to x^{z}$ as $p_z$. Its Mellin transform is $\MM(p_z)(s)=\frac{z}{s-z}$; whence using $\MM(\psip)=\MM(p_1)+\frac1{r}\MM(p_r)+2M\MM(p_{1/2})- \srs\frac{1}{\rho} \MM(p_\rho)+\MM(g)$ we are led to
\begin{align}\label{eq:zetaPcomputation}
\frac{\zeta'_{\PP}}{\zeta_{\PP}}(s)&=
-\frac{1}{s-1} -\frac{1}{s-r}-\frac{M}{(s-1/2)} + \srs \frac{1}{s-\rho} - G(s),
\end{align}
where $G(s)$ is the Mellin transform of $g$, and as such, is analytic for $\Re s>1/2$.

Consider the product
\begin{equation}\label{eq:Qdef}
Q(s):=\frac{1}{s-1} \frac{1}{s-r} \frac{1}{(s-1/2)^{M}}\prod_{\rho\in\SSS} \left({s}-{\rho}\right).
\end{equation}
It is clear that $\frac{\zeta'_{\PP}}{\zeta_{\PP}}(s)=\frac{Q'}{Q}(s)-G(s)$. Therefore
\begin{equation}\label{eq:zetafirstproduct}
\zeta_{\PP}(s)= c~ Q(s) H(s) \qquad \textrm{with}\quad  H(s):=e^{-\int_1^s G(z) dz} \quad \textrm{and}\quad c\ne 0 \quad\textrm{a constant}.
\end{equation}
By construction, $H(s)$ is analytic and nonvanishing for $\Re s >1/2$ and so is $\zeta_{\PP}$ except for a simple pole at $s=1$ and possibly another simple pole at $s=r$ (if $r$ exceeded $1/2$), and zeroes $\rho\in\SSS$ with exactly the multiplicity given in $\SSS$. Let us underline that $\zeta_{\PP}$ vanishes nowhere else in the halfplane $\Re s>1/2$ of meromorphic continuation.

\medskip
Next, we are to show that the integer counting function $\N(x)$, generated by our set $\PP$ of primes, will satisfy Axiom A. To succeed, we will need the full strength of \eqref{pitestimateDMVZ} or \eqref{pitestimate}, which we formulate with the use of the function
$$
J(x,t):=\sum_{p_j\le x} p_j^{-it} - \int_1^x y^{-it} d F(y)=\int_1^x y^{-it} d(\pip(y)-F(y)).
$$
With this, the above theorems say that we have
\begin{equation}\label{eq:Jfromthms}
|J(x,t)| \ll \sqrt{x} + \sqrt{x\frac{\log(|t|+1)}{\log(x+1)}}.
\end{equation}
So far the analytic characteristics of $\zp$ were found, but we also need good order estimates. Note that for $\Re s\ge 3/2$ we definitely have $|\zp(s)| \le \zp(3/2)<\infty$, whence the order of magnitude of $\zp$ is under some control. However, to proceed we need a different, more precise control on the size of $\zp$, valid also in the critical strip. Equivalently, we estimate
\begin{align}\label{logzeta}
\log\zp(s)&=\sum_{p\in \PP} \log \left(\frac{1}{1-p^{-s}}\right)=\int_{1}^\infty x^{-s} d\PiP(x) =\MM(\PiP)(s),
\end{align}
where $\PiP(x)$ is the Riemann modified prime counting function--coming to picture in view of the Euler product formula \eqref{Euler}--and having the exact form
$$
\PiP(x):=\sum_{p_j\in\PP;~ p_j^n\le x} \frac{1}{n}=\sum_{n=1}^{[\log x/\log p_1]}\frac{\pip(x^{1/n})}{n}=\int_1^x \frac{d \psip(u)}{\log u}.
$$
For large enough $\Re s$ by absolute convergence we can write
\begin{align}\label{Pdevelopment}
Z(s):=\log \zp(s)&=\int_1^\infty x^{-s} d\left(\sum_{n=1}^\infty \frac1{n} \pip(x^{1/n})\right)=\sum_{n=1}^\infty \frac1{n} \int_1^\infty x^{-s} d\pip(x^{1/n}) \notag
\\  &=\sum_{n=1}^\infty \frac1{n} \int_1^\infty y^{-ns} d\pip(y)= \sum_{n=1}^\infty \frac1{n} \MM(\pip)(ns).
\end{align}
Put
$$
P(s):=\MM(\pip)(s)=\int_1^\infty x^{-s} d\pip(x).
$$
If $\Re s\ge 3/2$, then we have
$$
|P(s)|\le \int_1^\infty x^{-s} d\pip(x) \le \int_1^\infty x^{-\si} d\PiP(x) =\log\zp(\si)\le \zp(\si)-1 =\sum_{g\in \GG;~ g\ne 1} \frac{1}{|g|^{\si}}.
$$
Given that $\Re ns \ge 3/2$ for any $s$ with $\Re s\ge 1/2$ and $n\ge 3$, we therefore can write
$$
\left|\sum_{n=3}^\infty \frac1{n} P(ns) \right| \le \sum_{n=3}^\infty \sum_{g\in \GG;~ g\ne 1} \frac{1}{|g|^{n\si}} =\sum_{g\in \GG;~ g\ne 1}\frac{|g|^{-3\si}}{1-|g|^{-\si}} \le \frac{\zp(3/2)}{1-|p_1|^{-1/2}},
$$
furnishing
\begin{equation}\label{Pexplicit}
Z(s)=P(s)+\frac12 P(2s) + P^*(s)\qquad (\Re s >1) \qquad \textrm{with}\quad |P^*(s)|=O(1)~ (\Re s \ge 1/2).
\end{equation}
Note that here the function $P^*(s)$ extends analytically and boundedly to the closed halfplane $\Re s \ge 1/2$, too, even if $Z(s)$ is analytic--and the formula itself is shown--only for $\Re s>1$.

Next, we are to evaluate $P(s)$ for $\Re s\ge 1$. The main term will be provided by $L(s):=\MM(F)(s)$, with an error
$$
R(s):=\MM(\pip-F)(s)=\int_1^\infty x^{-s} d(\pip(x)-F(x)).
$$
Partial integration gives $R(s)=\left[x^{-\si} J(x,t)\right]_1^\infty +\si \int_1^\infty x^{-\si-1} J(x,t) dx=\si \int_1^\infty x^{-\si-1} J(x,t) dx$, so that \eqref{eq:Jfromthms} furnishes
\begin{equation}\label{eq:Rfirstesti}
|R(s)|\ll \frac{\si}{\si-1/2} + \si \frac{\sqrt{\log(|t|+2)}}{\sqrt{\si-1/2}} \quad (\Re s>1/2),
\end{equation}
showing in particular that $R(s)$ is analytic in the halfplane $\Re s>1/2$.

\begin{lemma}\label{l:powerperlog} Let $z \in \CC$ with $\Re z \le 0$ be arbitrary. Then the following integral formula holds.
\begin{equation}\label{powerperlog}
I(z,s):=\int_1^\infty \frac{x^{z-s}-x^{z-1-s}}{\log x} dx = \log\left( \frac{s-z}{s-z-1}\right)\quad (\Re s >1).
\end{equation}
\end{lemma}
\begin{proof} Obviously, the integral converges absolutely and uniformly in any halfplane $\Re s\ge \si_0$ with $\si_0>1$. Therefore, it gives an analytic function for $\Re s>1$. Also we have $\lim_{\Re s\to \infty} I(z,s)=0=\lim_{\Re s\to \infty} \log\left( \frac{s-1}{s-z-1}\right)$, therefore it suffices to check that the derivatives with respect to $s$ of the two analytic expressions match. Differentiation (executed below the integral sign for $I(z,s)$) and a little calculus afterwards yields the assertion.
\end{proof}

The Mellin transform $L(s):=\MM(F)(s)$ decomposes to similar expressions as in the Lemma. Indeed, using $dF(x)=\frac{1-1/x}{\log x}d f_0(x)=((1-1/x)f_0'(x)/\log x) dx $ we can write
\begin{align*}
L(s)&:=\MM(F)(s):=\int_1^\infty x^{-s} dF(x)= I(0,s)+I(r-1,s)+M I(-1/2,s)-\srs I(\rho-1,s)
\\ & =\log\left( \frac{s}{s-1}\right)+\log\left( \frac{s-r+1}{s-r}\right)+M\log\left( \frac{s+1/2}{s-1/2}\right)-\srs\log\left( \frac{s-\rho+1}{s-\rho}\right).
\end{align*}
Collecting the above yields
\begin{equation}\label{eq:logzetapssummedup}
Z(s)=L(s)+R(s) +\frac12 (L(2s)+R(2s))+P^*(s)=L(s)+\frac12 L(2s) +R^*(s),
\end{equation}
where $R^*(s):=R(s)+\frac12 R(2s)+P^*(s)$ is analytic for $\Re s>1/2$ and satisfies
\begin{equation}\label{eq:Rstar}
|R^*(s)| \ll \frac{\si}{\si-1/2} + \si \frac{\sqrt{\log(|t|+2)}}{\sqrt{\si-1/2}} \quad (\Re s>1/2).
\end{equation}
Exponentiating provides us the product representation
\begin{align}\label{eq:zetasproduct}
\zp(s)= \frac{s}{s-1} &\sqrt{\frac{s}{s-1/2}} \frac{s-r+1}{s-r} \sqrt{\frac{s-r/2+1/2}{s-r/2}}
\left( \frac{s+1/2}{s-1/2}\right)^{M} \left( \frac{s+1/4}{s-1/4}\right)^{M/2} \notag
\\\cdot  & \prod_{\rho\in \SSS} \left(\left(\frac{s-\rho}{s-\rho+1}\right)\sqrt{\frac{s-\rho/2}{s-\rho/2+1/2}} \right) \cdot e^{R^*(s)},
\end{align}
Although formally we have got this formula for large $\Re s$ only, in fact by meromorphic continuation it extends to all $\Re s>1/2$. In that halfplane $\zp(s)$ is seen to have a simple pole at $s=1$, and another simple pole at $s=r$ in case $r>1/2$, but no more singularity. Note in particular that all the root expressions are analytic in $\Re s>1/2$ (as well as the respective logarithms were), because these terms have some singularities only for $\Re s\le 1/2$.

Anyway, it is clear that $\N(x)$--the number of Beurling integers with norm (absolute value) not exceeding $x$ in the $\PP$-generated free semigroup--is an increasing function. Let us use its integral $\N_1(x)=\int_1^x \N(y)dy$, which is easier to handle for the inverse Mellin transform (the Perron integral expression) for that is absolutely and uniformly convergent--in view of $|\zp|\le \zp(3/2)$--for any $d\ge 3/2$, say:
$$
\N_1(x)=\frac{1}{2\pi i} \int_{d-i\infty}^{d+i\infty} \frac{x^{s+1}\zp(s)}{s(s+1)}ds=\frac{1}{2\pi i} \int_{d-i\infty}^{d+i\infty} \frac{x^{s+1}\exp\left(\log\zp(s) \right)}{s(s+1)}ds.
$$
Taking into account nonnegativity of $\N(x)$, we have
$$
\N_1(x-1)-\N_1(x) \le \N(x) \le \N_1(x+1)-\N_1(x)=:\DD(x),
$$
so that we will get through with a good asymptotic evaluation of $\DD(x)$. From the above
\begin{equation}\label{eq:DPerron}
\DD(x)=\frac{1}{2\pi i} \int_{3/2-i\infty}^{3/2+i\infty} \frac{((x+1)^{s+1}-x^{s+1})\zp(s)}{s(s+1)}ds.
\end{equation}
Here we introduce two parameters $1/2<a<T$, $a\approx 1/2$ and $T$ large in terms of $x$, and deform the contour of integration to the line $\Re s=a$. Denoting $\kappa:=\lim_{s\to 1} (s-1)\zp(s)$ and in case $r>1/2$ also $\lambda:=\lim_{s\to r} (s-r)\zp(s)$, we get by the Residuum Theorem
$$
\DD(x)=\frac{1}{2\pi i} \int_{(a)} \frac{((x+1)^{s+1}-x^{s+1})\zp(s)}{s(s+1)}ds + \frac{\kappa}{2}(2x+1)+\frac{\lambda}{r(r+1)}\left((x+1)^{r+1}-x^{r+1}\right),
$$
provided that we had $r>1/2$ and $a$ was chosen to satisfy $1/2<r<a$, while the last term is simply missing if $r=1/2$. So let us agree that the value of $\lambda$ is as defined above if $r>1/2$, and is $0$ for $r=1/2$; also, we will assume $1/2<a$ even if $r=1/2$, so that the above formula remains valid even for $r=1/2$.

We cut up the vertical line for the integral to the two parts with $|t|\le T$ and $|t|>T$. For the first we estimate $|((x+1)^{s+1}-x^{s+1})|=|(s+1)\int_x^{x+1} y^s dy|\le |s+1| (x+1)^a<|s+1|2\exp(a\log x)$. So taking $a:=1/2+\delta$, where $\de:=\de_x:=\log^{-1/3}x$ we get on this part $|((x+1)^{s+1}-x^{s+1})|\le 2|s+1| \sqrt{x} \exp(\log^{2/3} x)$. Now, let us assume also $x>x_0(r)$ (in case $r>1/2$) to guarantee $\de<\frac12(r-1/2)$ in this case. For $\zp(s)$ we estimate the terms in \eqref{eq:zetasproduct} utilizing that the factors in the first row are maximal either for $t=0$ or for $t\to \infty$ whenever $s=a+it$, and that the terms in $\prod_{\SSS}$ are all below $1$; also, for the last, exponential factor we make use of \eqref{Rstar}. These yield with some constant\footnote{$C$ is considered here and everywhere a generic constant not necessarily the same at each occurrence.} $C$
$$
|\zp(s)|\le \left| \frac{s}{s-1}\right| 3^{2+M+M/2}\de^{-M-2} e^{O\left( \frac{1}{\de} + \frac{\sqrt{\log(|t|+2)}}{\sqrt{\de}}\right)}\le \left| \frac{s}{s-1}\right| e^{C\log^{1/3}x+C\log^{1/6}x\log^{1/2} \max(2,t)}.
$$
It follows that on the line segment $[a-iT,a+iT]$ we have
\begin{align*}
\bigg| \frac{1}{2\pi i} \int_{a-iT}^{a+iT} & \frac{((x+1)^{s+1}-x^{s+1})\zp(s)}{s(s+1)}ds \bigg| \le \int_{a-iT}^{a+iT} \left|\frac{((x+1)^{s+1}-x^{s+1})}{s+1}\right| \left|\frac{\zp(s)}{s}\right|~dt
\\ & \le 2 \sqrt{x} ~ e^{\log^{2/3}x}~ e^{C\log^{1/3}x+C\log^{1/6}x\log^{1/2}T} ~\frac{1}{\pi} \int_0^T \frac{dt}{|(a+it)-1|}
\\ & \ll \sqrt{x} ~ e^{\log^{2/3}x+C \log^{1/3} x + C\log^{1/6}x\log^{1/2}T}~\log T.
\end{align*}
The estimation on the infinite parts of the line $\Re s=a$ is similar, but there we use only $|((x+1)^{s+1}-x^{s+1})|\le 2(x+1)^{a+1} \le 3x^{a+1} = 3 x^{3/2} e^{\log^{2/3}x}$, so that here we obtain
\begin{align*}
\bigg| \frac{1}{2\pi i} \int_{a+iT}^{a+i\infty} & \frac{((x+1)^{s+1}-x^{s+1})\zp(s)}{s(s+1)}ds \bigg|
\le \int_{a+iT}^{a+i\infty} x^{3/2} e^{\log^{2/3}x} \left|\frac{\zp(s)}{s(s+1)}\right|dt
\\ & \le x^{3/2} e^{\log^{2/3}x} \int_T^\infty \frac{e^{C\log^{1/3} x+ C\log^{1/6}x\log^{1/2} t} dt}{|(a+it)^2-1|}
\\& \le x^{3/2} e^{\log^{2/3}x} \int_T^\infty \frac{e^{(C+C^2)\log^{1/3}x+\frac14 \log t}}{t^2}dt~\le ~\frac{x^{3/2}}{T^{3/4}}\cdot e^{C'\log^{2/3}x}.
\end{align*}
Taking $T=x^2$, say, we easily obtain from the estimates for the two parts that in fact
$$
\bigg| \frac{1}{2\pi i} \int_{(a)} \frac{((x+1)^{s+1}-x^{s+1})\zp(s)}{s(s+1)}ds \bigg| \le \sqrt{x} \exp(O(\log^{2/3} x)),
$$
actually with any implied $O$-constant which exceeds the constant in \eqref{eq:Rfirstesti} by more than 1. Winding up, we are led to
$$
\DD(x)=\kappa x + \frac{\lambda }{r}x^r +O\left(\sqrt{x} \exp(O(\log^{2/3} x)) \right)
$$
and as $\DD(x-1)\le \N(x) \le \DD(x)$, we even get
$$
\N(x)=\kappa x + \frac{\lambda }{r}x^r +O\left(\sqrt{x} \exp(O(\log^{2/3} x)) \right).
$$
Note that by definition here $\lambda=0$ if $r=1/2$, but for any other value $r>1/2$ we have $\lambda \ne 0$, because we have seen that $\zp(s)$ vanishes only at the points $\rho \in \SSS$, all having $\beta=\Re \rho >r$.

So it follows that in case $r$ exceeded $1/2$ we have Axiom A with $\theta=r$, but with no smaller value, and similarly if $r=1/2$ then we have Axiom A with all value $1/2+\ve$. Unfortunately, we do not obtain from this argument--whose natural limit is at the order of $\sqrt{x}$, in view of the applied random prime selection algorithm theorems--whether the arising $\N(x)$ can satisfy Axiom A with a possibly smaller value of $\theta$, or at least with $\theta=1/2$ itself.

\begin{theorem}\label{thm:Beurlingconstruction} Let $1/2\le r<1$ be a parameter and $\SSS$ a finite, symmetric (w.r.t. the real axis) multiset of elements $\rho=\beta+i\gamma$, all satisfying $r<\beta=\Re \rho <1$.

If $r>1/2$, then there exists a sequence $\PP$ of Beurling primes such that the corresponding integer counting function satisfies Axiom A with $\theta=r$, while Axiom A is satisfied with no smaller value in place of $\theta=r$.

Further, if $r=1/2$, then there exists a sequence $\PP$ of Beurling primes such that the corresponding integer counting function satisfies Axiom A with $\theta=1/2+\ve$, for any $\ve>0$.

Moreover, the prime number formula $\psip(x)\sim x$ holds with an error term $\Dp(x):=\psip(x)-x$ satisfying $\Dp(x)=\sum_{\rho \in \SSS} - \frac{x^\rho}{\rho}+O(\sqrt{x})$.

Furthermore, the Beurling zeta function is analytic in the halfplane $\Re s>1/2$, except for a simple pole at $s=1$ with residuum $\kappa$, and in case $r>1/2$ with another simple pole at $s=r$, and in this halfplane it vanishes precisely at the points of $\SSS$, with the multiplicity prescribed in the multiset.
\end{theorem}

\begin{remark}\label{rem:trick} It should be clear from the construction that the $O(\sqrt{x})$ error term of $\psi_{\PP}$ was brought about only by the formula  $\psi_{\PP}(x)=\sum_{n=1}^{[\log x]} \vartheta(x^{1/n})$, for $\vartheta(x)$ holding a much more precise error estimate of the order of $\log x$. Making preliminary adjustments (subtracting $\sqrt{x}$ from $f_0(x)$ when defining $\vartheta(x)$) this can be pushed down: in fact, to arbitrary given $0<q<1/2$ we can make corresponding adjustments in this formula to ensure $\psi_{\PP}(x)=f_0(x)+O(x^q)$ with an appropriate $f_0(x)$. Therefore, a proper adjustment of the above example works for arbitrary $0<q<1$, providing a zeta function with $(s-1)\zp(s)$ having a maximal halfplane of analyticity $\Re s >\theta_0=q$ and $\zp(s)$ regular and nonzero in $\Re s >q$ except for the given $\rho_k$ and their "retracts" $\rho_k/2$ etc. (due to the appearance of $\vartheta{(\sqrt{x})}$ etc.), plus the simple pole at $s=1$. However, even if we know analyticity, we cannot get good order estimates in the "bad half" $q<\Re s \le 1/2$ of the critical strip, and, correspondingly, we cannot infer a better error estimate for $\N(x)$.

Although this may seem only a deficiency of our methods, in fact it is not. Indeed, take $\SSS:=\emptyset$. Then the above would mean $\Dp(x)=O(x^q)$. However, according to a result of Neamah and Hilberdink \cite{H-20} about so-called $(\alpha,\beta,\gamma)$ systems of arithmetical semigroups, among the three constants--representing the "best exponent" in the error term for the PNT, the best $\theta$ in Axiom A for $\N(x)$, and the best exponent for the asymptotic error in the formula for the sum of the M\"obius function--the two largest have to be equal and exceed $1/2$. Therefore, if PNT holds with a good error bound, i.e., with an error exponent $q<1/2$, then necessarily the "M\"obius exponent" and "the best $\theta$ in Axiom $A$" are both at least $1/2$ (and match). This shows that we cannot expect, in general, to arrive at Axiom A with better exponent than $1/2$, at least not when we can prescribe good error bounds for the PNT.
\end{remark}

\section{Proof of Theorem \ref{th:Binterference}}\label{sec:proofinterference}

Here we will combine the above considerations--the sine polynomial $S$ presented in Lemma \ref{l:finitesum} and the Beurling number system constructed in Theorem \ref{thm:Beurlingconstruction}--to prove Theorem \ref{th:Binterference}.

Let $\ve>0$ be given and take $S$ as in Lemma \ref{l:finitesum}. Let a parameter $v>0$ be chosen large enough with respect to conditions soon to follow. Take, further, $\rho_k:=\beta_0+i(2n_k+1)v$ for all $k=0,1,\ldots,N$, with the sequence $(n_k)$ coming from the terms of $S$; in particular, let $\rho_0=\beta_0+iv$, so that finally we will set $\gamma_0=v$. Similarly, we will write $\gamma_k:=(2n_k+1)v$. Let us write $\alpha_k:=\arctan(\gamma_k/\beta_0)=\arctan((2n_k+1)v/\beta_0)~(k=0,1,\ldots,N)$. With a slight abuse of notation, we take $\SSS:=\cup _{k=0}^N \{\rho_k, \overline{\rho_k}\}$ (so that here $\#\SSS=2N+2$, not $N$).

An application of Theorem \ref{thm:Beurlingconstruction} provides a Beurling system of primes $\PP$ and integers $\N$ such that Axiom A holds with $\theta=r$ (and with no smaller value, if $r$ was $>1/2$), and such that
$$
\Dp(x)=-\sum_{\rho \in \SSS} \frac{x^\rho}{\rho}+O(\sqrt{x})=-2 x^{\beta_0} \sum_{k=0}^N \frac{\cos(\gamma_k\log x -\alpha_k)}{|\rho_k|}+O(\sqrt{x}).
$$
Let us compute $\Dp(x)|\rho_0|x^{-\beta_0}$. Discarding a negligible $O(x^{1/2-\beta_0})$ term arising from the $O(\sqrt{x})$ above, and writing in $y:=\log x$, this is approximately
$$
T(y):=-2 \sum_{k=0}^N \frac{\cos(\gamma_k y -\alpha_k)}{|\rho_k/\rho_0|} \qquad (y:=\log x).
$$
First, we replace the $\cos$ terms by $\sin(2(n_k+1)vy)$ here, using that $|\alpha_k-\pi/2|=\arctan\left(\frac{\beta_0}{(2n_k+1)v}\right)\le 1/v$ in view of $0<\beta_0<1$. Thus the total error from this approximation can be estimated as
$$
2\sum_{k=0}^N \left| \frac{\cos\left((2n_k+1)v y -\alpha_k \right) }{|\rho_k|/|\rho_0|} - \frac{\sin((2n_k+1)v y)}{|\rho_k|/|\rho_0|} \right| \le \frac{2N+2}{v}.
$$
Second, we modify the denominators from $|\rho_k|/|\rho_0|$ to $(2n_k+1)$ which again results in an error not exceeding
\begin{align*}
\bigg|\sin((2n_k+1)v y) & \left(\frac{|\rho_0|}{|\rho_k|} - \frac{1}{2n_k+1}\right)\bigg| \le \frac{(2n_k+1)|\rho_0|-|\rho|_k}{|\rho_k|(2n_k+1)}
\\&= \frac{(2n_k+1)^2|\rho_0|^2-|\rho|_k^2}{|\rho_k|(2n_k+1)((2n_k+1)|\rho_0|+|\rho_k|)}
< \frac{(2n_k+1)^2\beta_0^2}{(2n_k+1)^2|\rho_0| |\rho_k|} < \frac{1}{v}.
\end{align*}
Summing up, we have
$$
\left| T(y) -S(vy)\right| = \left| T(y)  - 2 \sum_{k=0}^N \frac{\sin((2n_k+1)vy)}{(2n_k+1)}\right| \le \frac{(4N+4)}{v } \le \ve,
$$
provided we choose $v>(4N+4)/\ve$. It follows that for large enough $v$ we have
$$
\left|\Dp(x) \frac{|\rho_0|}{x^{\beta_0}} - S(v\log x) \right| \le \ve + O(x^{1/2-\beta_0}) \le 2\ve,
$$
if $x>x_0(\ve)$. Thus we are led to
$$
\left|\Dp(x)\right| \le\left(\|S\|_\infty +2 \ve\right) \frac{x^{\beta_0}}{|\rho_0|} \le\left(\pi/2+3 \ve\right) \frac{x^{\beta_0}}{|\rho_0|},
$$
whence the theorem.

\section{Concluding remarks and preview of further work}\label{sec:conclusion}

Denote by $\eta(t):(0,\infty)\to (0,1/2)$ a nonincreasing function
and consider the domain
\begin{equation}\label{eq:etazerofree}
\DD(\eta):=\{ s=\sigma+it \in\CC~:~ \sigma>1-\eta(t),~ t>0\}.
\end{equation}
Following Ingham \cite{Ingham} and Pintz \cite{Pintz1,Pintz2}
we will then use the derived function\footnote{Note that $\omega_\eta$ can be expressed via the Legendre transform of the function $\eta(e^v)$, see \cite{Rev-D}.}
\begin{equation}\label{omegadef}
\omega_{\eta}(x):=\inf_{y>1} \left(\eta(y)\log x+\log y\right).
\end{equation}

\begin{theorem}[Pintz]\label{th:domainesti} Assume that there is no
zero of the Riemann $\zeta$ function in $\DD(\eta)$. Then for arbitrary $\ve>0$
we have
$$
\Delta(x)=O(x\exp(-(1-\ve)\omega_\eta(x)).
$$
\end{theorem}

\begin{theorem}[Pintz]\label{th:domainosci}
Conversely, assuming that there is an infinite sequence of zeroes within
the domain \eqref{eq:etazerofree}, we have for any $\ve >0$
the oscillation $\Delta(x)=\Omega(x\exp(-(1+\ve)\omega_\eta(x))$.
\end{theorem}

These results, in their original proofs and / or sharpest forms
relied on particular things generally not available for the Beurling
zeta functions. Therefore, it was unclear how much of these relations
can as well be stated for the distribution of Beurling primes?

In our forthcoming work \cite{Rev-Many} we generalize the above results of Pintz to the Beurling case. That is, we prove\footnote{{\bf Note added in proof}: The versions here are extended variants of what is originally stated in the work \cite{Rev-Many}, where Conditions B and G played a role in view of an application of Theorem \ref{th:density} referring to these conditions. The sharpening comes from a very recent result in \cite{Rev-NewDens}, not yet refereed, where
this Carlson-type zero density theorem was extended to all Beurling systems satisfying Axiom A. For details see \cite{Rev-NewDens}.} the following.

\begin{theorem}\label{th:Beurlingdomainesti} Let the arithmetical semigroup $\GG$ satisfy Axiom A, and assume that $\eta(t):(0,\infty)\to (0,1-\theta)$ is a nonincreasing function. If $\DD(\eta)$ is free of zeroes of the Beurling zeta function $\zp$, then for arbitrary $\ve>0$ we have $\Dp(x)=O(x\exp(-(1-\ve)\omega_\eta(x))$.
\end{theorem}

\begin{theorem}\label{th:Beurlingdomainosci}
Conversely, if $\eta(t)$ is also convex in logarithmic variables (i.e., $\eta(e^v)$ is convex in $v$), and there are infinitely many zeroes of $\zp$ within the domain \eqref{eq:etazerofree}, then we have for any $\ve >0$
the oscillation $\Dp(x)=\Omega(x\exp(-(1+\ve)\omega_\eta(x))$.
\end{theorem}

Such type of general results, although known for the classical Riemann case or e.g. for algebraic number fields, but are rare in the Beurling context. An interesting contribution is \cite{Johnston}. The work there can be interpreted as an investigation about the very special zero-free region bounded by
$$
\eta(t):=\begin{cases} 1/2 \qquad&\textrm{if}\quad |t|\le T, \\ 0 \qquad&\textrm{if}\quad |t|> T,
\end{cases}
$$
that is, assuming RH for $|t|\le T$, and nothing for $|t|>T$. Although very special as for $\eta$, it still has importance, because numerical information on the validity of RH over a finite range already yield effective error estimates for PNT. Even if Johnston presents his work for the Riemann zeta only, the proofs are general and can as well be applied for the Beurling case. Similarly, an explicit error estimate of $\Delta(x)$ can always be obtained from knowledge of a de la Vallée Poussin-type zero-free region \eqref{classicalzerofree}, which, in turn, can be obtained by century old methods of de la Vallée Poussin and Landau even in the Beurling case--unlike more sophisticated results for zero-free regions, which do not extend to the Beurling setup. Our point is the precise connection of constants: if \eqref{classicalzerofree} is known with a concrete $c$, then we also want a concrete constant in \eqref{classicalerrorterm}. (E.g. in \cite{MT} the constants obtained for the Riemann case are $1/5.573412$ and $1/\sqrt{6.315}$.) Our goal is to extend such special results to general ones, which tell, in the generality of the Beurling setup and in an essentially optimal way in each direction, how the information on zero-free regions is connected to error bounds for PNT.

\section{Funding acknowledgements}

Supported in part by Hungarian National Foundation for Scientific Research, Grant \# T-72731,
T-049301, K-61908, K-119528 and K-132097 and by the Hungarian-French Scientific and
Technological Governmental Cooperation, Project \# T\'ET-F-10/04 and the Hungarian-German Scientific and Technological Governmental Cooperation, Project \# TEMPUS-DAAD \# 308015.

\end{document}